\documentclass[a4paper,11pt]{amsart}%
\usepackage{hyperref}
\usepackage{pdfsync}
\usepackage{dsfont}
\usepackage{color}
\usepackage{amsthm}
\usepackage{enumerate}
\usepackage{amsfonts}
\usepackage{amssymb}%
\usepackage{mathtools}
\usepackage{braket}
\usepackage{bm}
\usepackage{amsmath}
\usepackage{graphicx}
\usepackage{caption}
\usepackage{fullpage}
\usepackage{pb-diagram}
\numberwithin{equation}{section}
\newtheorem{theorem}{Theorem}[section]

\newtheorem{definition}[theorem]{Definition}
\newtheorem{lemma}[theorem]{Lemma}

\newtheorem{proposition}[theorem]{Proposition}

\theoremstyle{remark}
\newtheorem{remark}[theorem]{Remark}
\newtheorem{example}[theorem]{Example}

\newcommand{\R}{\mathds{R}}
\newcommand{\N}{\mathds{N}}

\newcommand{\X}{\mathcal{X}}

\newcommand{\converges}[1]{ \overset{#1}{\longrightarrow}} 
\newcommand{\M}{\mathcal{M}}
\newcommand{\cut}{Cut}
\newcommand{\x}{\mathbf{x}}
\newcommand{\bnu}{\bm{\nu}}

\newcommand{\veps}{\varepsilon}

\DeclareMathOperator{\dist}{dist}

\DeclareMathOperator{\divergence}{div}

\DeclareMathOperator{\Lip}{Lip}

\definecolor{mygreen}{rgb}{0.1,0.75,0.2}

\newcommand{\nc}{\normalcolor}

\title{ Variational limits of $k$-NN graph-based functionals on data clouds}
\author{ Nicol\'as Garc\'ia Trillos}

\newcommand{\Addresses}{{
  \bigskip
  \footnotesize
  \textsc{Division of Applied Mathematics, Brown University,
    Providence, RI, 02912, USA. }\par\nopagebreak  \textit{Email:} \texttt{nicolas\_garcia\_trillos@brown.edu}
  }}

\begin{document}
\maketitle

\begin{abstract}
This paper studies the large sample asymptotics of data analysis procedures based on the optimization of functionals defined on $k$-NN graphs on point clouds. The paper is framed in the context of minimization of balanced cut functionals, but our techniques, ideas and results can be adapted to other functionals of relevance. We rigorously show that provided the number of neighbors in the graph  $k:=k_n$ scales with the number of points in the cloud as $n \gg k_n \gg \log(n)$, then with probability one, the solution to the graph cut optimization problem converges towards the solution of an analogue variational problem at the continuum level.
\end{abstract}

\providecommand{\keywords}[1]{\textbf{\textit{Keywords---}} #1}
\keywords{k-NN graph, discrete to continuum limit, Gamma-convergence, graph cut, Cheeger cut, spectral clustering.}

\section{Introduction}

This paper studies the large sample asymptotics of data analysis procedures based on the optimization of functionals defined on graphs; the procedures of interest include graph-based methods for clustering, classification, and semi-supervised learning. The set of vertices of the graph is a random data set $X_n:=\{ \x_1, \dots, \x_n \}$  while the edges capture the level of similarity among points. In this work our focus is on $k$-NN graphs, where one puts an edge between a pair of points whenever one of them is among the $k$-nearest neighbors of the other; we will assume from here on that $X_n$ is a subset of Euclidean space.  Our main results rigurously show that when $k$ scales like $\log(n) \ll k \ll n$, then the solutions of the optimization problems of interest at the graph level are consistent and converge towards the solutions of analogue variational problems at the continuum level.  Spectral clustering, total variation clustering, diffusion maps, and $p$-Laplacian regularization, are all examples of graph-based data analysis procedures with a variational flavor (see \cite{bluv13,coifman1,HeinBuhl,HeinSetz,RanHein,ShiMalik,tutorial,Ramdas}), and the results and ideas in this paper can be applied to analyze their large sample limit in this $k$-NN setting. For expository purposes we focus on the concrete example of minimizing balance graph cuts as described below.

To get a flavor of our main results, consider a set of random points uniformly distributed on the region $D$ as depicted in Figure \ref{fig3}. A $k$-NN graph for a certain choice of $k$ is then obtained as shown in Figure \ref{fig4}. We introduce a functional on partitions $\{A, A^c \}$ of $X_n$ by 
 \begin{equation}
\cut_{n}(A, A^c) := \frac{\sum_{\x_i \in A}\sum_{\x_j \not \in A} w_{ij}  }{ \min \{ \lvert A \rvert, \lvert A^c \rvert \} } , \quad A \subseteq X_n.
\label{CheegerCut}   
\end{equation}
 The numerator favors low interaction between the two sets in a partition, while the denominator forces it to be balanced in terms of size. It is thus natural to consider the optimization problem
 \begin{equation*}
   \min_{A \subseteq X_n} \cut_{n}(A, A^c)   
  \end{equation*}
as a sensible approach for data clustering;  \eqref{CheegerCut} is known as the \textit{Cheeger cut} functional.  In Figure \ref{fig5} we illustrate the minimizer of \ref{CheegerCut} for the graph in Figure \ref{fig4}. We observe the close resemblance between the discrete minimizer in Figure \ref{fig5} and the partition of the region $D$ in Figure \ref{fig6} which can be described as a solution to a variational problem at the continuum level of the form
\[ \min_{A\subset D} \cut(A,A^c), \] 
where the cut functional $\cut(A,A^c)$ is defined (at least for $A$ with smooth boundary) as
\[\cut(A,A):= \frac{\int_{ \partial A \cap D } dS  }{\min \{ |A|, |A^c| \}}. \]
\begin{figure}[!htb]
	\minipage{0.4\textwidth}
	\includegraphics[ trim=5cm 9cm 0cm 7cm, clip=true ,width=1.5\linewidth]{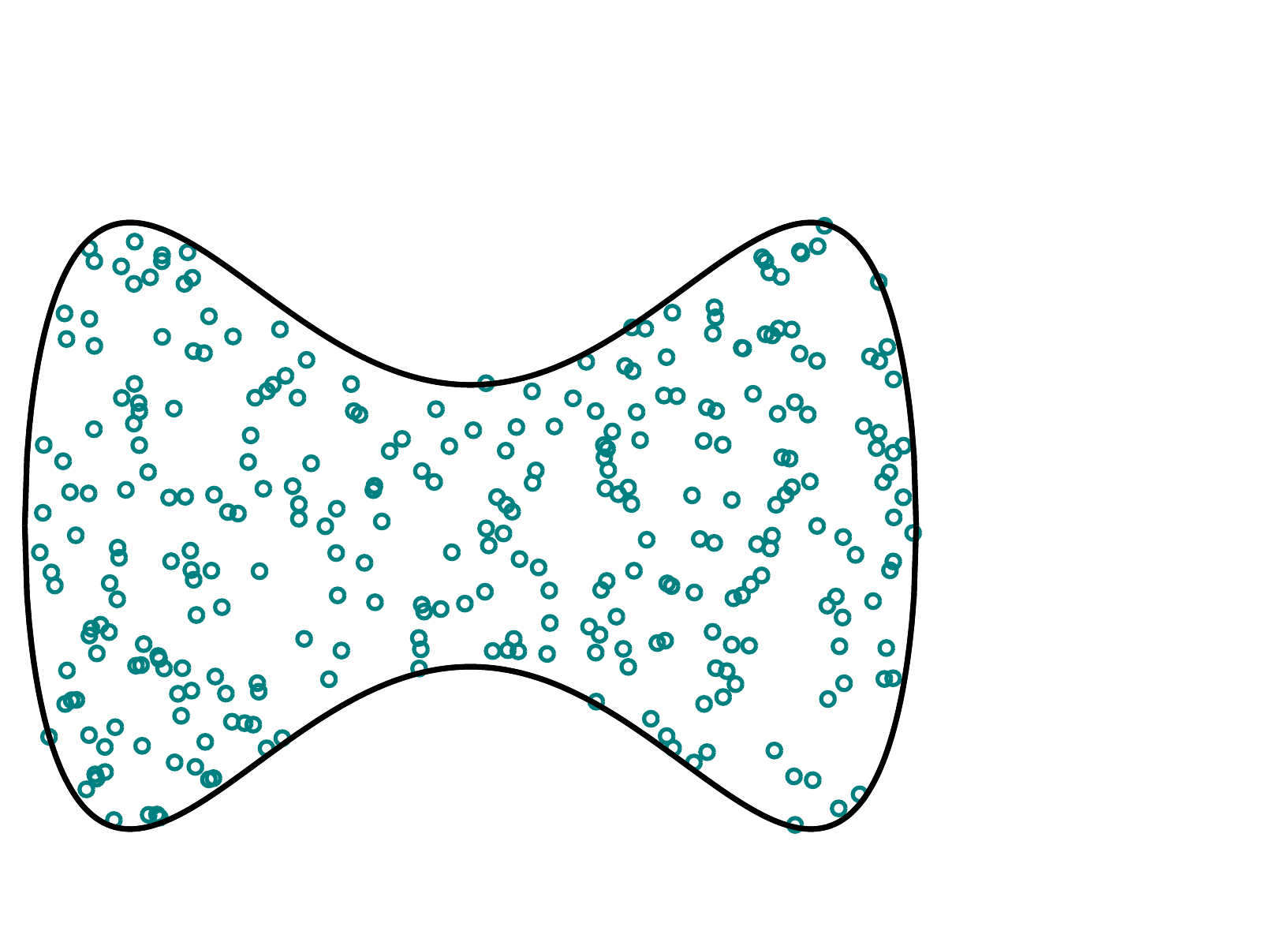}
	\caption{A sample of $n=120$ points.}\label{fig3}
	\endminipage \hfill
	\minipage{0.4\textwidth}
	\includegraphics[trim=5cm 9cm 0cm 7cm, clip=true,width=1.5\linewidth]{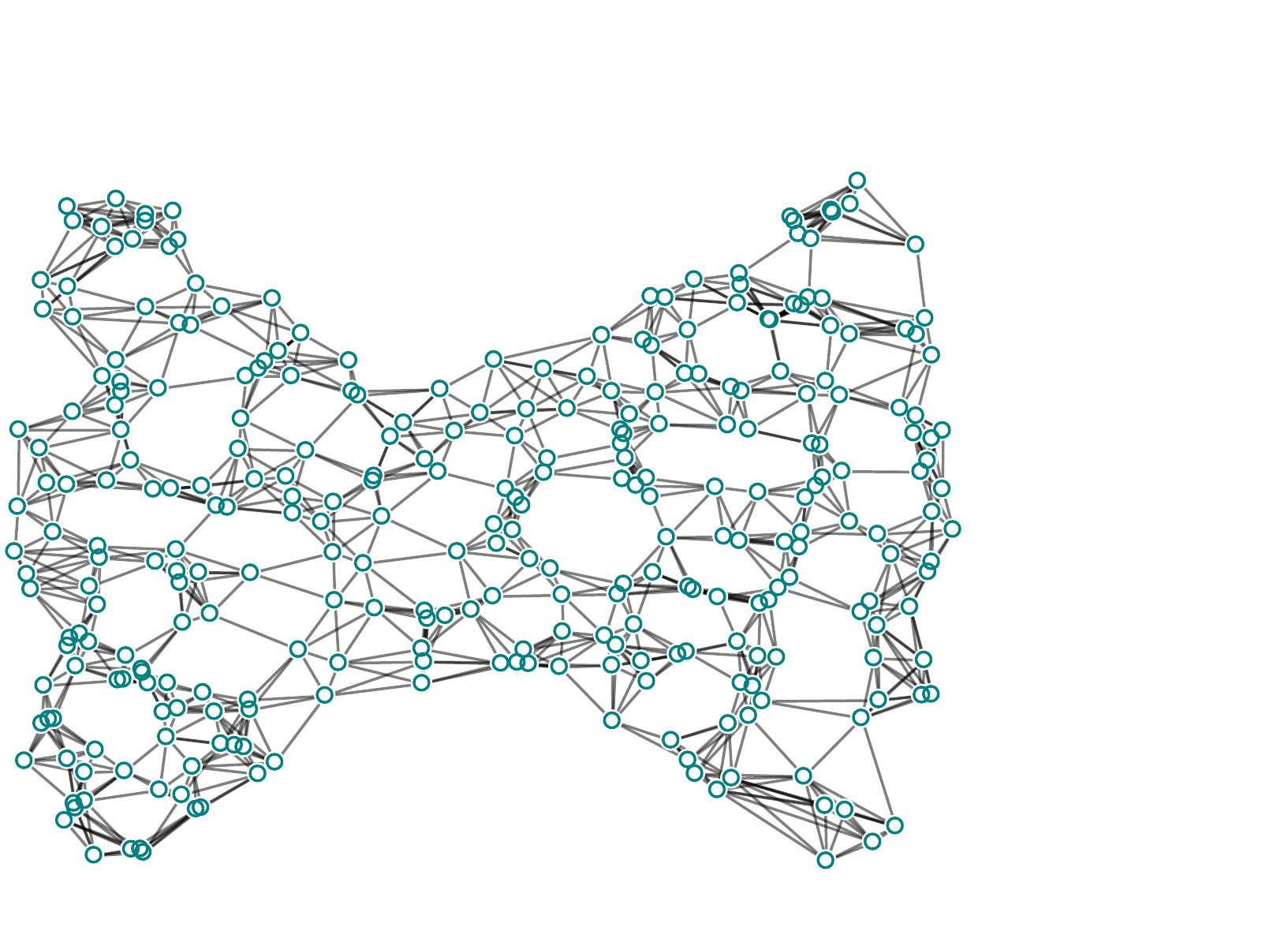}
	\caption{Geometric graph with $k=6$.}\label{fig4}
	\endminipage
\end{figure}
\begin{figure}[!htb]
	\minipage{0.4\textwidth}
	\includegraphics[trim=5cm 9cm 0cm 7cm, width=1.5\linewidth]{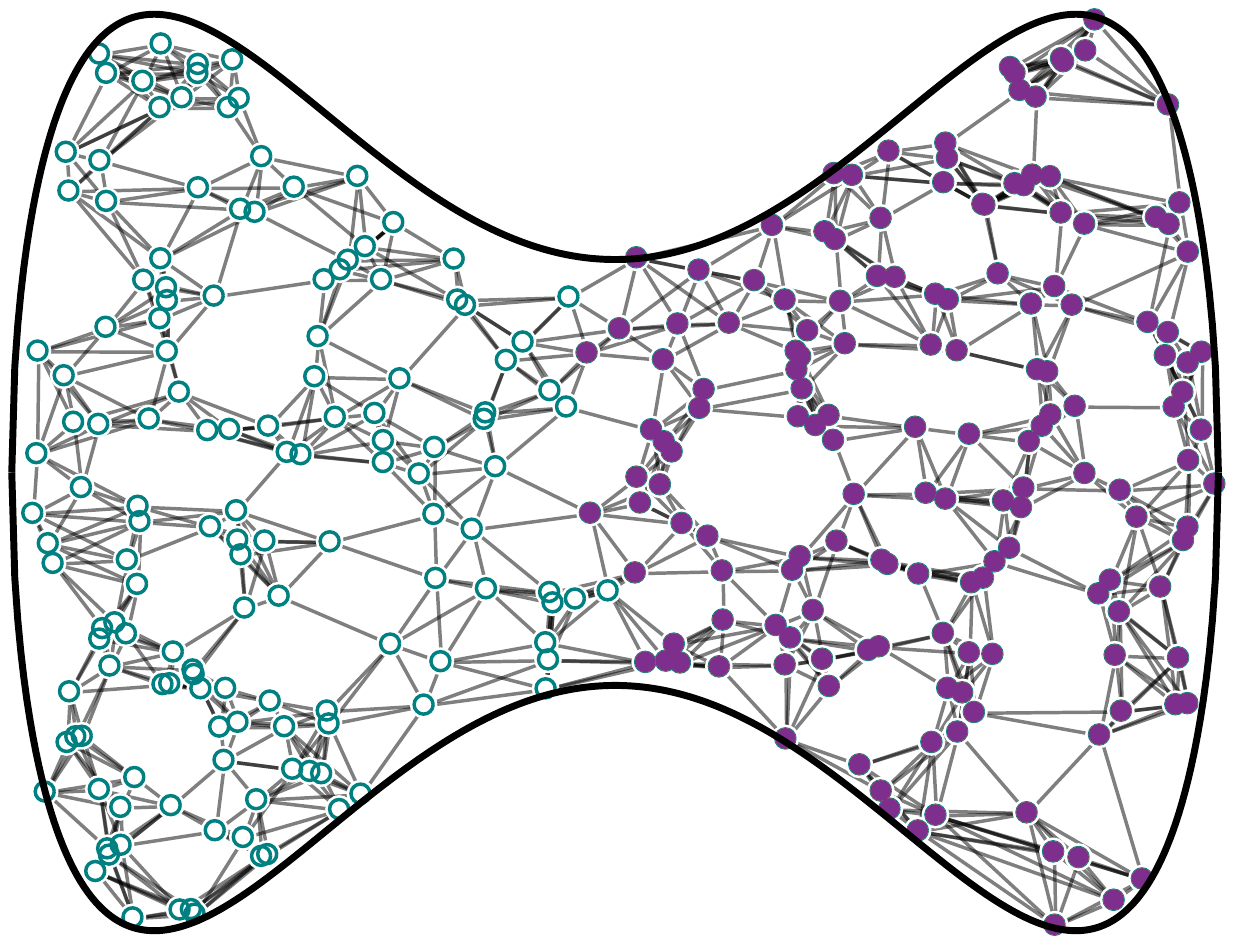}
	\caption{Minimizer of Cheeger cut.}\label{fig5}
	\endminipage\hfill
	\minipage{0.4\textwidth}
	\includegraphics[ trim=3cm 1.5cm 0cm 0cm clip=true,width=1.3\linewidth]{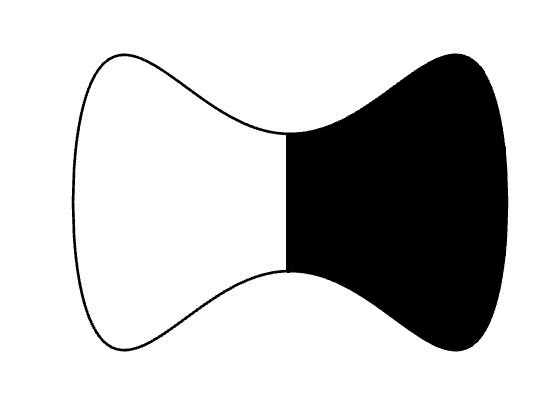}
	\caption{Minimizer continuous problem}\label{fig6}
	\endminipage\hfill
\end{figure}
The main theorem of this paper is a rigorous mathematical statement of the previous observation. We show that with probability one, and in a very precise sense, the solution to \eqref{CheegerCut} converges as $n \rightarrow \infty$ towards the solution to a continuum variational problem (see Theorem \ref{main0} below), provided that $k$ (the number of neighbors in the definition of the graph) scales like 
\[ \log(n) \ll k \ll n .\]
Phrased in the language of data clustering, our results establish the statistical consistency of a series of clustering procedures based on minimization of graph cuts when these come from $k$-NN graphs. Moreover, our results show that (at least in terms of scaling with $n$), the condition on $k$ needed to establish the consistency is dimension free.
\begin{figure}[!htb]
	\minipage{0.4\textwidth}
	\includegraphics[ trim=6.5cm 10cm 0cm 7cm, clip=true ,width=1.7\linewidth]{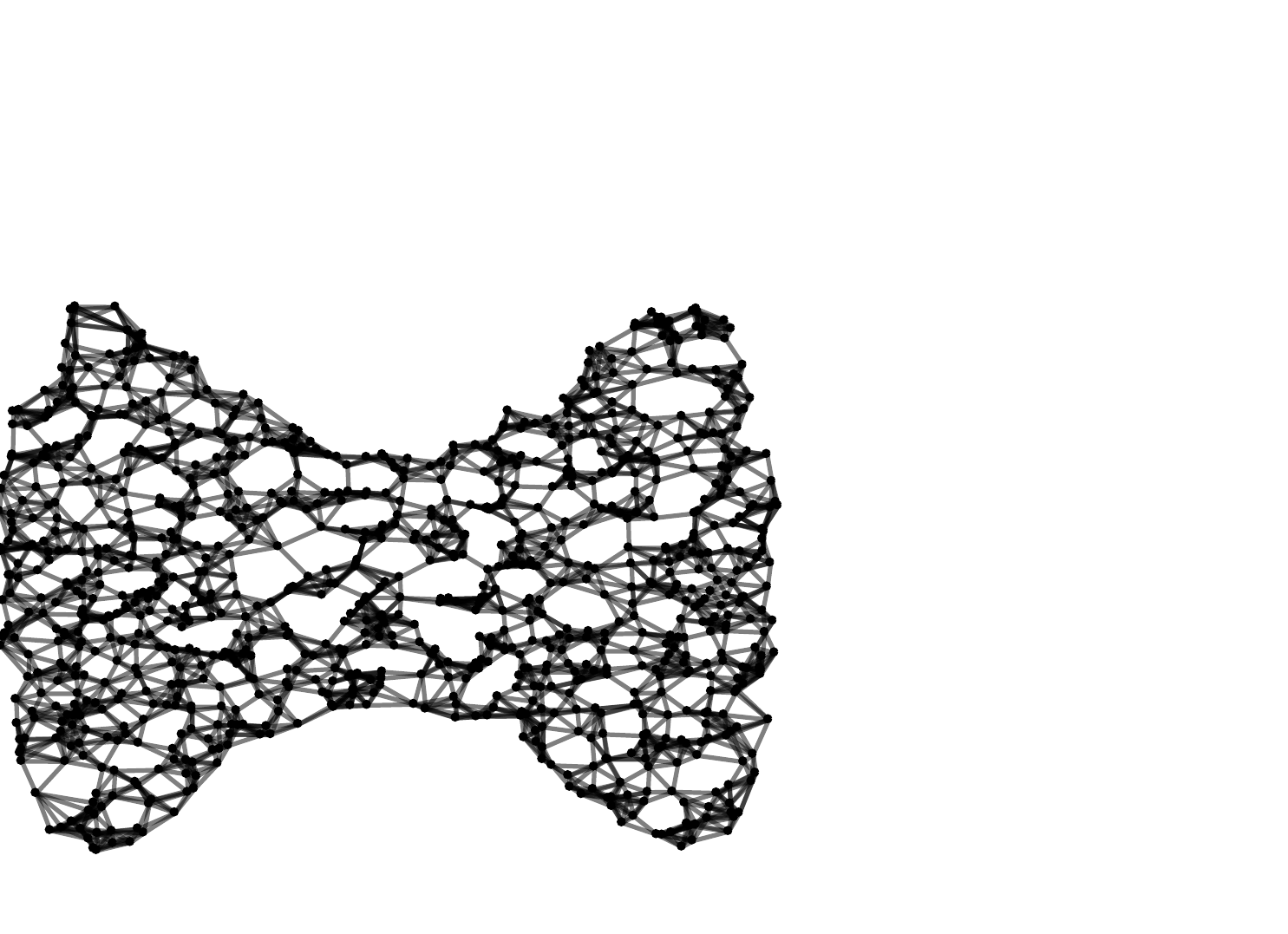}
	\endminipage \hfill
	\minipage{0.4\textwidth}
	\includegraphics[trim=6.5cm 10cm 0cm 7cm, clip=true,width=1.7\linewidth]{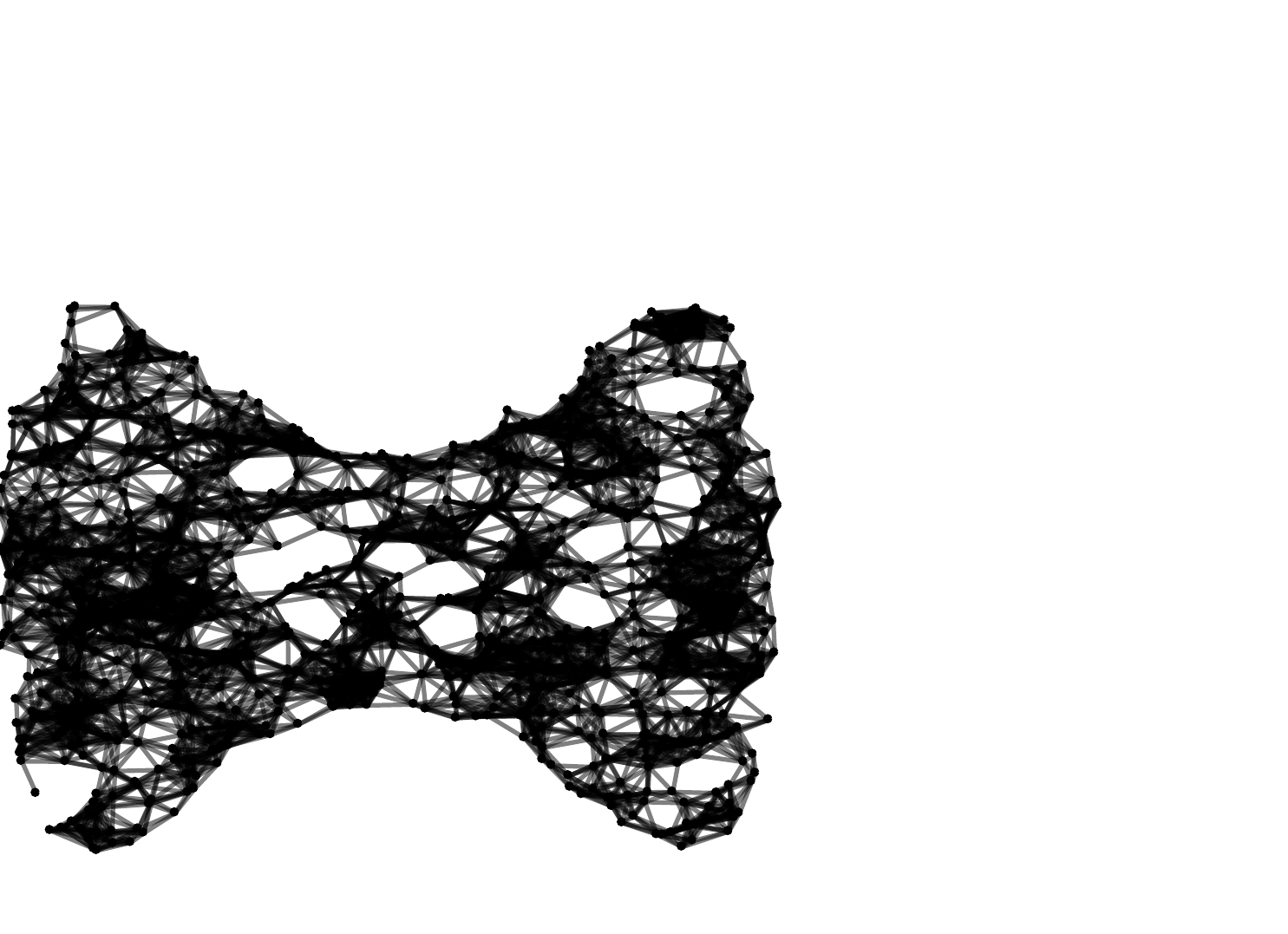}
	\endminipage
	\caption{The value of $\veps$ for the $\veps$-graph was chosen so that the graph was connected. We notice that the resulting graph exhibits less regularity than its $k$-NN counterpart.}
	\label{kNNepsilon}
\end{figure} 
  
To the best of our knowledge this work is the first one to rigorously address the stability of \textit{variational} problems on $k$-NN graphs such as the one defined in \ref{CheegerCut} in the large sample limit. Most of the theoretical works found in the literature addressing similar questions assume an $\veps$-\textit{graph} construction on the data set (i.e. there is an edge between two points if they are within distance $\veps$ of each other); an exception to this is the work \cite{THJ} where \textit{pointwise} convergence of graph Laplacians on $k$-NN graphs (among other constructions) is analyzed. We find the absence of theoretical results in the $k$-NN setting to be a strong motivation for our work given their frequent use by practitioners due to their nicer regularity properties and their robustness to data dimensionality; see \cite{tutorial} for a more complete discussion on this matter. Notice that $k$-NN graphs can be constructed completely from ordinal information about the data points and in particular exact values of interpoint distances are not needed (this is a property that adds to the robustness of the $k$-NN construction). Figure \ref{kNNepsilon} illustrates the difference between $\veps$-graphs and $k$-NN graphs. 

In this paper we follow the same line of thought described in \cite{TSvBLX_jmlr}, and in particular, reduce our analysis to obtaining the $\Gamma$-limit (see definitions in section \ref{Prelim}) of a rescaled version of the energy
\[ \sum_{\x_i \sim_k \x_j } \lvert  u(\x_i) - u(\x_j) \rvert,  \]
defined for functions $u_n : X_n \rightarrow \R$; this is the functional that appears in the numerator in \eqref{CheegerCut} (restricting to $u= \mathds{1}_A$), we will denote it by $GTV$ and refer to it as the \textit{graph total variation}, see \eqref{GraphTV} below for its 
precise definition.  The $\Gamma$-limit of $GTV$ is shown to be a weighted (local) total variation functional at the continuum level.  The notion of $\Gamma$-convergence provides precise sufficient conditions implying the stability of minimizers of variational problems; we review its definition in Section \ref{Prelim} and we refer the interested reader to \cite{DalMaso} for a more complete discussion on the topic.

Most of the technical work in this paper is devoted to analysing the ``bias'' of the random functional $GTV$. We establish the $\Gamma$-convergence of a kernel based functional with inhomegeneous bandwidth (a sort of average of $GTV$) towards a (local) weighted total variation (see Propositions \ref{Lemma1} and \ref{LemmaAux} below). The inhomogeneity associated to the kernel based functional is intrinsic to the $k$-NN graph construction, where length-scales are determined by the Euclidean distance and the data density around each point .  The fact that the resulting kernel based approximation has a varying bandwidth makes our analysis different from that in \cite{GarciaTrillos2015}. At each point in space, the bandwidth depends inversely on the density of the ground-truth measure generating the data. 

It is possible to reinterpret this feature and think of it as fixing a bandwidth but now measuring distances with an effective metric induced by the ground-truth on the ambient space; this metric in particular shrinks distances on regions with low density. In this work we will not pursue this idea any further, but we anticipate that there are several advantages of doing so. In particular, it is of relevance to further investigate the dimension free condition $\log(n) \ll k \ll n$, 
and understand better the constants appearing in these asymptotic inequalities. This is of special relevance if we want to extend our results to settings where the ground-truth distribution is for example a Gaussian measure in an finite dimensional Hilbert space. Both the unboundedness of the support of the ground truth and the infinite dimension of the ambient space are settings that are not covered by the results in this paper. We believe that a better understanding of this and other related issues are of importance for a better understanding of $k$-NN graphs and their benefits.

We would like to finish this introduction by mentioning some of the growing literature on large sample asymptotic analysis of operators and functionals constructed from $\veps$-graphs.  Convergence of graph Laplacians can be found in the work by Belkin and Niyogi \cite{bel_niy_LB},   
Coifman and Lafon \cite{coifman1},
Gin\'e and Koltchinskii \cite{GK}, Hein, Audibert and von Luxburg \cite{hein_audi_vlux05}, and 
Singer \cite{Singer}. These works deal mostly with pointwise consistency of graph Laplacians. The work of Arias-Castro, Pelletier and Pudlo \cite{Pelletier} studies the pointwise convergence of Cheeger energy and consequently of total variation, as well as variational convergence when the discrete functional is considered over an admissible set of characteristic functions which satisfy a ``regularity" requirement.  Spectral convergence of graph Laplacians, (relevant for data clustering) has been studied, among others,  by Huang, and  Jordan \cite{THJ}, Belkin and Niyogi \cite{belkin2007convergence}, 
von Luxburg, Belkin and Bousquet \cite{vonLuxburg}, Singer and Wu \cite{SinWu13}. Most of the results previously mentioned are deduced using tools from perturbation theory of linear operators.  A different set of tools was introduced in \cite{GarciaTrillos2015} and later used in \cite{GarciaMurray,SpecCon,TSvBLX_jmlr}.  The notion of $\Gamma$-convergence (a.k.a. epi-convergence) and the introduction of a suitable metric (the $TL^1$ metric; see \cite{GarciaTrillos2015} and Section \ref{Prelim} below) 
allowing to compare functions at the graph level with functions at the continuum level, 
were crucial tools to deduce statistical consistency for a large class of balanced graph cuts  \cite{TSvBLX_jmlr} and for spectral clustering \cite{SpecCon}. In all these results, sharp convergence rates for $\veps:=\veps_n$ guaranteeing the consistency were provided. In the context of graph-based approaches to semi-supervised learning, it is worth mentioning the work \cite{ThorpeSlepcev} where the $p$-Laplacian regularization (for $p$ large enough) is studied, as well as the Bayesian approach in \cite{GraphBayes,GTASK} where the convergence of graph posteriors is analyzed.

\nc

\subsection{Outline} The rest of the paper is organized as follows. In Section \ref{setup} we make our assumptions precise, present some of the examples of variational problems that are relevant for important tasks in data analysis, and present the main results of the paper. In Section \ref{Prelim} we give the definitions of the $TL^1$ space and the notion of $\Gamma$-convergence; we also present some auxiliary results that are needed in the remainder. Finally, in Section \ref{Proofs} we present the proofs of the main results.

\section{Set-up and main results}
\label{setup}

We assume that the data $X_n:=\left\{ \x_1, \dots, \x_n \right\}$ are  i.i.d samples from a distribution $\nu$ on $\R^d$. $\nu$ is assumed to be an absolutely continuous measure with respect to the Lebesgue measure, with density $\rho: D \rightarrow \R$ where $D$ is a bounded, connected, open set with Lipschitz boundary, and $\rho$ is a continuous function bounded above and below by positive constants. Namely, we assume that there are positive constants $0< \rho_{min} < \rho_{max}$ such that for every $x \in D$ we have
\begin{equation}
  \rho_{min} \leq \rho(x) \leq \rho_{max}, \quad \forall x \in D.
\label{rhominmax}
\end{equation}
We denote by $\nu_n$ the empirical measure
\[ \nu_n:= \frac{1}{n}\sum_{i=1}^n \delta_{\x_i}, \] 
and write 
\[ \x_i \sim_k \x_j \]
whenever $\x_i$ is among the $k$-nearest neighors of $\x_j$ or vice-versa.

\begin{remark}
\label{Remknn}
Because $\nu$ has a density, with probability one, the condition $\x_i \sim_{k} \x_j$ is equivalent to
\[ \nu_n \left( B(\x_i, r) \right) \leq \frac{k}{n}\quad  \text{ or }  \quad \nu_n \left( B(\x_j, r) \right) \leq \frac{k}{n},  \]
where $r= \lvert \x_i - \x_j \rvert$. Without the loss of generality we assume this fact in the remainder.
\end{remark}

let us introduce some functions that we use in the remainder. Let $\eta: [0,\infty) \rightarrow \R$ be the Heaviside function given by 
\[ \eta(t) := \begin{cases}1  & \text{ if } t<1 \\ 0 & \text{ if } t \geq 1, \end{cases} \]
For an arbitrary $\veps>0$ we let $\eta_\veps$ be the function
\[   \eta_\veps(t) := \frac{1}{\veps^d}\eta\left( \frac{t}{\veps} \right), \quad t \in [0, \infty).\]
We let $\sigma_\eta$ be the quantity
\begin{equation}
\sigma_\eta := \int_{\R^d} \eta(z)|z_1|dz,
\label{sigma}
\end{equation}
where $z_1$ is the first coordinate of the vector $z$.  Notice that in the above expression $z_1$ can be replaced with $z \cdot v$ for any vector $v$ with unit norm.

\subsection{Graph total variation and total variation in the continuum}
In order to introduce the \textit{graph total variation} functional $GTV$ (an appropriate rescaled version of the numerator in \eqref{CheegerCut}), it will be convenient to define
\[ J_k(\x_i, \x_j):= \begin{cases} 1, \quad \x_i \sim_k \x_j \\ 0, \quad \text{otherwise} \end{cases}  \]
and for $k:= k_n \in \N$, let $\bar{\veps}_n$ be the number for which
\begin{equation}
 \bar{\veps}_n^d = \frac{k_n}{n}. 
 \label{hateps}
\end{equation}
The \textit{graph total variation} functional is then defined as
\begin{equation}
 GTV_{n,k}(u_n):= \frac{1}{n^2 \bar{\veps}_n^{d+1}} \sum_{i,j}J_{k}(\x_i, \x_j)  \lvert u_n(\x_i) - u_n(\x_j)  \rvert , \quad u_n \in L^1(\nu_n).
 \label{GraphTV}
\end{equation}
Notice that when $u_n = \mathds{1}_{A_n}$ for some $A_n \subseteq X_n$,  $GTV_{n,k}(u_n)$ is just a rescaled version of the denominator of the $\cut_n$ functional in \eqref{CheegerCut}.
At the discrete level we are interested in the following optimization problem
 \begin{equation}
\min_{A_n \subseteq X_n} \frac{GTV_{n,k}(\mathds{1}_{A_n}) }{\frac{1}{n} \min \{ \lvert A \rvert, \lvert A^c \rvert \} }.
\label{CheegerCut2}   
\end{equation}

The continuum counterpart of the graph total variation is a functional that takes the following form.
\begin{definition}
Let  $h: D \rightarrow \R$ be a continuous function bounded below and above by positive constants. For an arbitrary function $u\in L^1(D)$, we define its weighted (by $h$) total variation as
\begin{equation} 
TV(u; h) :=  \sup \left\{    \int_{D} u(x) \divergence(\zeta)dx \: : \: \zeta \in C^\infty_c(D: \R^d),  \quad \lvert \zeta(x) \rvert \leq h(x) \quad \forall x \in D    \right\}.
\label{TV}
\end{equation}
In the above and in the remainder of the paper, we use $L^1(D)$ to represent the space of $L^1$ functions with respect to the Lebesgue measure restricted to $D$.  In addition, when $h \equiv 1$ we write $TV(u)$ instead of $TV(u;h)$ and we denote by $BV(D)$ the space of functions $u \in L^1(D)$ for which $TV(u) < \infty$. Notice that for any $h$ continuous and bounded above and below by positive constants, the condition $TV(u)< \infty$ is equivalent to the condition $TV(u; h) < \infty$.
\end{definition}

\begin{remark}
If  $u \in BV(D)$ is smooth, then 
\[ TV(u;h)= \int_{D} \lvert  \nabla u \rvert h(x) dx.  \]
Also, if $u = \mathds{1}_A$ for some open set $A$ with smooth boundary, then 
\[ TV(u;h) = \int_{\partial A \cap D} h(x) d\mathcal{H}^{d-1}(x), \]
where $\mathcal{H}^{d-1}$ is the $(d-1)$-dimensional Hausdorff measure.
\end{remark}

At the continuum level we will be interested in a variational problem of the form
\begin{equation}
\min_{A \subseteq D} \frac{TV(\mathds{1}_A; h)}{\min \{ \nu(A), \nu(D \setminus A) \}},
\label{CheegerCut2Cont}
\end{equation}
for an appropriately chosen function $h$.

\subsection{Main results}

Our main result establishes the convergence of minimizers of \eqref{CheegerCut2} towards minimizers of \eqref{CheegerCut2Cont}. Notice however that solutions to \eqref{CheegerCut2} are discrete sets, whereas solutions to \eqref{CheegerCut2Cont} are continuum sets, and so we need to clarify the sense in which we will establish the convergence of discrete minimizers towards continuum ones. The convergence is taken in the $TL^1$ space introduced in \cite{GarciaTrillos2015}, where in particular functions on the point cloud and functions on $D$ are seen as elements of the same space (see section \ref{Prelim} below for details).

We are ready to establish our main result.

\begin{theorem}
	\label{main0}
Let $d \geq 3$  and let $D \subseteq \R^d$ be an open, connected, bounded domain with Lipschitz boundary.  Let $\rho: D \rightarrow \R$ be a continuous density function satisfying \eqref{rhominmax} and let $\nu$ be the probability measure $d \nu= \rho dx$. Let $\left\{ k_n \right\}_{n \in \N}$ be a sequence of natural numbers satisfying
\[ \lim_{n \rightarrow \infty} \frac{k_n}{\log(n)} =+\infty, \quad \lim_{n \rightarrow \infty}  \frac{k_n}{n}=0. \]
Let $\{ \x_1, \dots, \x_n\}$ be i.i.d. points from $\nu$ and let $A_n^*$ be a solution to \eqref{CheegerCut2}.

Then, with probability one, along every subsequence of $\{A_n^*\}_{n \in \N}$ there is a further subsequence converging in the $TL^1$-sense towards a minimizer of \eqref{CheegerCut2Cont}, where 
\[ h(x):= \rho^{1-1/d}(x).  \]

Moreover, the minimum value of \eqref{CheegerCut2} converges as $n$ goes to infinity towards $\sigma_eta/ \alpha_d^{1+1/d}$ times the minimum value of \eqref{CheegerCut2Cont}, where $\sigma_\eta$ is defined in \eqref{sigma} and $\alpha_d$ is the volume of the unit ball in $\R^d$. 
\end{theorem}

\begin{remark}
In the above theorem if the minimizer $\{A^*, A^{*c} \}$ of \eqref{CheegerCut2Cont} is unique, then the convergence is along the entire sequence of discrete minimizers.
\end{remark}

As discussed in the introduction, and especially as shown in \cite{TSvBLX_jmlr}, our results are a direct corollary of Theorems \ref{Main} and \ref{Compac} below. We show that the $\Gamma$-limit (in the $TL^1$-sense) of the functional $GTV_{n,k_n}$ (for the same scaling $k=k_n$ as in Theorem \ref{main0}) is the functional $TV(\cdot ; \rho^{1-1/d})$. The notion of $\Gamma$-convergence and its connection to the stability of minimizers of functionals are reviewed in section \ref{Prelim}.

\nc

\begin{theorem}($\Gamma$-convergence)
Let $d \geq 3$  and let $D \subseteq \R^d$ be an open, connected, bounded domain with Lipschitz boundary.  Let $\rho: D \rightarrow \R$ be a continuous density function satisfying \eqref{rhominmax}. 
Let $\left\{ k_n \right\}_{n \in \N}$ be a sequence of natural numbers satisfying
\[ \lim_{n \rightarrow \infty} \frac{k_n}{\log(n)} =+\infty, \quad \lim_{n \rightarrow \infty}  \frac{k_n}{n}=0. \]
Then, the functionals $GTV_{n, k_n}$ $\Gamma$-converge towards $\frac{\sigma_\eta}{\alpha_d^{1+1/d}} TV(\cdot ; \rho^{1-1/d}) $ in the $TL^1$-sense, where $\sigma_\eta$ is defined in \eqref{sigma}, and $\alpha_d$ 
is the volume of the unit ball in $\R^d$.  Furthermore, when restricted to indicator functions, the energies $GTV_{n,k_n}$ 
$\Gamma$-converge to the functional $\frac{\sigma_\eta}{\alpha_d^{1+1/d}} TV(\cdot ; \rho^{1-1/d}) $ restricted to indicator functions.
\label{Main}
\end{theorem}

\begin{theorem}(Compactness)
Under the same assumptions in Theorem \ref{Main}, the following statement holds with probability one:
If $\left\{  u_n\right\}_{n \in \N}$  is a sequence with $u_n \in L^1(\nu_n)$ for which
\[ \sup_{n \rightarrow \infty} \lVert u_n \rVert_{L^1(\nu_n)} < \infty   \]
and
\[ \sup_{n \rightarrow \infty} GTV_{n,k_n}(u_n) < \infty, \]
then $\left\{ u_n \right\}_{n \in \N}$ is pre-compact in $TL^1(D)$. That is, every subsequence of $\left\{ u_n \right\}_{n \in \N}$ has a further subsequence converging in $TL^1(D)$.
\label{Compac}
\end{theorem}

\begin{remark}
The above theorems hold for domains $D$ in $\R^2$ provided we replace the condition $ \lim_{n \rightarrow \infty} \frac{k_n}{\log(n)}= \infty$ with the condition $ \lim_{n \rightarrow \infty} \frac{k_n}{(\log(n))^{3/2}}= \infty $ . 
This can be seen directly from our proofs, and follows from the fact that the rate of convergence of the $\infty$-transportation distance between $\nu$ and $\nu_n$ (for $d=2$) scales like $\frac{\log(n)^{3/4}}{n^{1/2}}$ and not like $\frac{(\log(n))^{1/2}}{n^{1/2}}$ (see \cite{W8L8}).    
\end{remark}

\begin{remark}
\label{RemManifold}
The above set-up and results can be extended to the setting in which the support of $\nu$ is not a domain $D$ contained in the ambient space $\R^d$ but a compact manifold $\M \subseteq \R^d$ with intrinsic dimension $m$. 
When that is the case, the appropriate scalings for the functionals are obtained using the intrinsic dimension $m$ and not the dimension of the ambient space $d$. 
Although in this paper we omit the details of such extension, we point out that the results in \cite{Ponce} (later used in the proofs in \cite{GarciaTrillos2015}) can be adapted 
to the manifold setting by establishing results analogue to those in \cite{W8L8} in the manifold setting, using the geodesic flow in $\M$, and the fact that the Euclidean distance (distance in the ambient space $\R^d$) is a third order approximation for the geodesic distance (intrinsic distance), i.e. ,
\[ d_{\M}(x,y) = \lvert x-y \rvert + O(\lvert x-y \rvert^3), \quad  x,y \in \M.\]
The details can be presented elsewhere in a more general setting which is also of interest; see Remark \ref{Remdimension} below.
\end{remark}

\begin{remark}
\label{Remdimension}
Suppose that $d \geq 3$. One of the important consequences of Theorem \ref{Main} is that the admissible regimes for $k=k_n$ that guarantee the recovery of a non-trivial variational limit for the energies $GTV_{n,k}$
does not depend on $d$ (we impose $n \gg k_n \gg \log(n)$). As a consequence, notice that despite the fact that the scaling factor appearing in $GTV_{n,k}$ depends on $d$, 
the minimizers of the Cheeger cut are unaffected by a rescaling of the energy by a positive constant. In other words, in principle we do not need to know the dimensionality of the data before hand in order to obtain good clusters using the graph total variation associated to a $k$-NN graph. Notice that the same is not true for $\veps$-graphs. In particular,  it is plausible to extend the consistency of Cheeger cuts to the setting in which the dimensionality of the data changes in space, i.e., if $\nu$'s support is something like a union of manifolds with different intrinsic dimensions. This problem will be explored in the future.
\end{remark}

%


\subsection{Other examples of relevant variational problems on graphs}

In this section we present a  variety of examples of other optimization problems on graphs that are relevant for data analysis. The common structure in all these optimization problems is that in their objective functions the highest order term is either a graph total variation or a $L^p$ version of it; for this reason theorems \ref{Main} and \ref{Compac} (in conjunction with Proposition \ref{porpGamma}) are of relevance beyond the setting of Theorem \ref{main0}.

For the rest of this section $(w_{ij})_{ij}$ represents a similarity graph on a data set (not necessarily a $k$-NN graph).

%
%

\begin{example}(Ratio graph cuts for clustering). A functional closely related to the Cheeger cut in \eqref{CheegerCut} is the \textit{ratio cut} defined by
	\[ \cut_{R}(A, A^c) := \frac{\sum_{\x_i \in A}\sum_{\x_j \not \in A} w_{ij}  }{  \lvert A \rvert \cdot \lvert A^c \rvert } , \quad A \subseteq X_n. \]
	Both the above functional and the Cheeger cut can be seen as examples of a larger class of functionals known as \textit{balanced cuts}. 
	This type of functionals penalize partitions $\{A, A^c \}$ of $X_n$ that either have a big interface separating $A$ and $A^c$ or are highly unbalanced in the sense that $A $ and $ A^c $ are of dissimilar size; 
	intuitively, these are desirable features for a good partitioning of the data and motivate considering the minimization problem
	\[ \min_{A \subseteq X_n} \cut(A, A^c ),\]
	to obtain a good partition of $X_n$. Notice that the denominator in both functionals is, up to a multiplicative factor, the graph total variation defined by
	\[  GTV(\mathds{1}_A):= \sum_{i,j} w_{ij} \lvert  \mathds{1}_A (\x_i) - \mathds{1}_A(\x_j) \rvert. \]
	
	Assuming that the set of vertices consists of samples from the distribution $\nu$ and that the weights of the graph are those obtained from an $\veps$-graph 
	(with $\veps:= \veps_n$ chosen appropriately), 
	the results in \cite{TSvBLX_jmlr} state that minimizers of the Cheeger and ratio cuts converge, as $n \rightarrow \infty$, 
	towards solutions to analogue variational problems in the continuum.
	This result can then be interpreted as a consistency result for clustering using balanced cuts.
	\label{Exam1} 
\end{example}

\begin{example} (Graph total variation in the context of classification). Suppose that  
	\[\{ (\x_1,y_1), \dots , (\x_n,y_n) \}\] 
	are samples from some distribution $\bnu$ supported on $\overline{D} \times \{0,1 \}$.  
	The pair $(\x_i, y_i)$ is interpreted as the feature values ($\x $ variable) and label ($y$ variable) of an individual from a given population. 
	Based on the observed data (training data) the idea is to construct a ``good'' \textit{classifier} assigning labels to every potential individual in the population. A typical choice of risk functional used to define ``good'' classifiers is the \textit{average missclassification error}. Unfortunately, when the distribution $\bnu$ is unknown (as it is usually the case) 
	it is not possible to determine the classifier that minimizes the average misclassification error (Bayes classifier) and hence an approximation to it based exclusively on the observed data is the best thing one can hope for. 
	Such an approximation can be obtained using the graph total variation as we describe below. 
	
	Given the weighted graph $( \{ \x_i\}_i , (w_{ij})_{i,j})$, consider the energy
	\[  R_{n,\lambda}(u):=  R_n(u) + \lambda GTV(u), \quad u : \{ \x_1, \dots, \x_n \} \rightarrow \R .\]
	The functional $R_n$ is the \textit{empirical risk} and the parameter $\lambda$ is introduced so as to emphasize or deemphasize the regularizing effect of the graph total variation. In \cite{GarciaMurray}, the weights for the graph are assumed to be those coming from an $\veps$-graph 
	and the problem of obtaining an approximation to the Bayes classifier is divided into first solving the minimization problem
	\[ \min_{u: X_n \rightarrow \R } R_{n, \lambda}(u), \]
	and then extend the minimizer to the whole ambient space appropriately. With the right scaling for $\lambda=\lambda_n$, one can establish the asymptotic consistency of the constructed approximation. See \cite{GarciaMurray} for more details.
\end{example}

\begin{example}(Spectral clustering and spectral embeddings). 
	Undoubtedly, one of the most popular graph based methods for data clustering is \textit{spectral clustering}.  
	In the two way clustering setting, spectral clustering can be seen as a relaxation of the ratio cut minimization problem mentioned in Example \ref{Exam1}. Indeed, the relaxed problem takes the form
	\[  \min_{ u: X_n \rightarrow \R} \frac{ \sum_{i,j}w_{ij} ( u(\x_i) - u(\x_j) )^2  }{\sum_{i} ( u (\x_i) - \overline{u} )^2   }, \]
	where $\overline{u}$ is the average $\overline{u}:=\frac{1}{n}\sum_{i}u(v_i)$; the denominator of the above objective function is a $L^2$-version of the graph total variation. It is well known that the above optimization problem is actually an eigenvalue problem for the graph Laplacian and that the first non-trivial eigenvector of the graph Laplacian is a minimizer. In the context of multi-way clustering, higher eigenvectors of the graph Laplacian are used to define an embedding of the data cloud into a Euclidean space with low dimension. In turn, the embedded data can be clustered using an algorithm like $k$-means, inducing in this way a partition for the original data.  Among the many results in the literature addressing the consistency of spectral clustering (see the Introduction for some references), we highlight the work in \cite{SpecCon} which exploits the variational characterization of eigenvectors/eigenvalues. In that work the graph on the cloud $X_n$ is assumed to be an $\veps$-graph.   
	\label{ExamSpec}
\end{example}

\begin{example}($p$-Laplacian regularization for semi-supervised learning) Let 
\[\{ (\x_1,y_1), \dots , (\x_q,y_q) \}\]
be $q$ labeled data points and let 
\[ \x_{q+1}, \dots, \x_n \]
be a set of unlabeled data points. We think of $n$ as being much larger than $q$.

In \cite{Ramdas} the authors consider the optimization problem
\begin{equation}
\min_{u: X_n \rightarrow \R } \sum_{i,j}w_{ij}| u(\x_i) - u(\x_j)|^p  +  \frac{1}{2\gamma^2}\sum_{i=1}^q |u(\x_i)- y_i|^2
\label{semip}
\end{equation}
for the purposes of semi-supervised learning; here $p$ is a user chosen parameter, whose role is to impose regularity on candidate functions $u$ (the higher $p$ the more regular the functions will be). In two independent contributions \cite{Calder,ThorpeSlepcev} the authors study the large data limit of this variational problem and by studying the discrete regularity induced by the $L^p$ term in \eqref{semip}, they are able to rigurously show that there is a phase transition in the value of $p$ at which solutions to \eqref{semip} ``stop forgetting" the $q$ labeled data points; the transition occurs at $p> m$ (where $m$ is the intrinsic dimension of the $x$ data set), a result that is reminescent to Sobolev's embedding theorem. The setting in which these results are shown is that of $\veps$-graphs.
	
	\label{Examplesemi}
\end{example}

\begin{example}(Bayesian formulation of semi-supervised learning) In the context of Example \ref{Examplesemi}, a related optimization problem is  
\[ \min_{u: X_n \rightarrow \R } \langle \Delta_n^{\alpha} u , u \rangle +  \frac{1}{2\gamma^2}\sum_{i=1}^q |u(\x_i)- y_i|^2  ,\]
where $\Delta_n$ is the graph Laplacian associated to the graph $(w_{ij})_{ij}$ and where $\alpha$ is a positive number whose role is to enforce more or less regularity on a candidate function (higher $\alpha$ resuls in more regularity). The minimizer of this functional can be seen as the MAP of the \textit{posterior } distribution
\[ p(u|y) \propto \exp(- \phi(u; y) ) d \pi(u)  \]
where $\pi$ is a \textit{prior} distribution (in this case Gaussian with covariance matrix $\Delta_n^{-\alpha}$) and $\phi$ is a \textit{negative log-likelihood} model (in this case additive Gaussian noise). 

This Bayesian point of view to graph based semi-supervised learning was introduced in \cite{BertozziStuart}. In \cite{GraphBayes} the authors study the passage to the large $n \rightarrow \infty$ limit ($q$ fixed) and study the consistency of posterior measures. Moreover, from said consistency result and from the properties of the limiting posterior distribution, the authors in \cite{GTASK} provide some theory supporting the MCMC algorithm to sample from the posterior $p(u|y)$ that was proposed in \cite{BertozziStuart} . This algorithm was introduced so as to alleviate the curse of dimensionality when sampling from $p(u|y)$ (here the curse of dimensionality arises from the large number $n$). The setting in which all these results are shown is that of $\veps$-graphs.

\end{example}
\nc

\subsection{Preliminaries}
\label{Prelim}
The purpose of this subsection is to present some definitions and preliminary results that we use in the proof of our main theorems. In particular we present the definitions of $TL^1$-space and $\Gamma$-convergence.

$TL^1(D)$ is defined as the set of all pairs $(\mu, u  )$ where $\mu$ is a Borel probability measure on $D$ and $u$ is a function in $L^1(D,\mu)$ (written $L^1(\mu)$ from now on). 
This set can be endowed with the $TL^1$-metric defined by
\[ d_{TL^1}((\mu_1, u_1), (\mu_2, u_2) ):= \inf_{\pi \in \Gamma(\mu_1, \mu_2)} \left\{  \int_{D\times D}\lvert x-y \rvert d \pi(x,y) + \int_{D \times D} \lvert u_1(x) - u_2(y) \rvert d \pi(x,y) \right\}, \]
where $\Gamma(\mu_1, \mu_2)$ stands for the set of couplings between $\mu_1$ and $\mu_2$, that is, the set of all probability measures on the product $D\times D$ whose first and second marginals are $\mu_1$ and $\mu_2$ respectively.

Given that with probability one the empirical measure $\nu_n$ converges weakly to $\nu$ as $n \rightarrow \infty$, and given that $\nu$ has a density, we may use the characterization of $TL^1$-convergence from Proposition 3.12 
in \cite{GarciaTrillos2015} to conclude that $\{ (\nu_n, u_n) \}_{n \in \N}$ converges to $(\nu,u)$ in $TL^1$ if and only if \textit{there exists} a sequence of transportation maps 
$\{ T_n\}_{n\in \N}$ with ${T_n}_{\sharp} \nu = \nu_n$ ($T_n$ pushes forward $\nu$ into $\nu_n$), satisfying 
\begin{equation}
 \lim_{n \rightarrow \infty} \int_{D} \lvert T_n(x) - x \rvert d \nu(x) =0
 \label{Stag}
 \end{equation}
and 
\begin{equation}
 \lim_{n \rightarrow \infty} \int_{D} \lvert u_n(T_n(x)) - u(x) \rvert d \nu(x) =0. 
\label{Convrg}
\end{equation}
In turn, this holds if and only if \textit{for all} sequences of transportation maps with ${T_n}_{\sharp} \nu = \nu_n$ that satisfy \eqref{Stag} one has \eqref{Convrg}. Because of this characterization, 
we may abuse notation slightly and simply write $u_n \converges{TL^1} u$ without specifying the corresponding attached measures whenever it is clear from context that they can be omitted. 

A special choice of transportation maps between $\nu$ and $\nu_n$ that we use in the remainder are provided by \cite{W8L8}. We may consider transportation maps $T_n : D \rightarrow X_n$ between the measures $\nu$ and $\nu_n$ satisfying the condition 
\begin{equation}
 \lVert Id - T_n \rVert_{\infty}  \leq \frac{C (\log(n))^{1/d}}{n^{1/d}}=: \mathcal{\delta}_n,  
\label{Tn}
\end{equation}
for some constant $C$ (provided $d\geq 3$). Indeed, the results in \cite{W8L8} show that with probability one, 
there exists transportation maps $\left\{ T_n \right\}_{n \in \N}$ for which $T_{n \sharp} \nu = \nu_n$ and for which \eqref{Tn} holds for all large enough $n\in \N$. Since all of our results are asymptotic in nature, 
we may as well assume for the remainder that with probability one \eqref{Tn} holds for all $n \in \N$. One last relevant property of the map $T_n$, which follows directly from the fact that it transports $\nu$ into $\nu_n$, 
is the change of variables formula 
\begin{equation}
\int_{D} f(x) d \nu_n(x)=\int_{D}  f( T_n(x))d \nu(x) , 
\label{chvar}
\end{equation}
which allows us to write integrals with respect to $\nu_n$ in terms of integrals with respect to $\nu$.

We now present the definition of $\Gamma$-convergence in the context of a general metric space. This is a notion of convergence for functionals which together with a coercivity assumption, guarantee the stability of minimizers in the limit. A standard reference for $\Gamma$-convergence is \cite{DalMaso}.

\begin{definition} \label{def:Gamma}
Let $(\mathcal{X}, d_{\mathcal{X}})$ be a metric space and let  $F_n: \mathcal{X} \rightarrow[0,\infty] $ be a sequence of functionals. The sequence $\left\{ F_n \right\}_{n \in \mathbb{N}} $  $ \Gamma$-converges with respect to the metric  $d_{ \mathcal{X}}$ to the functional $F: \mathcal{X} \rightarrow  [0, \infty]$ as $n \rightarrow \infty$ 
if the following properties hold:
\begin{itemize}
\item \textbf{Liminf inequality:} For every $x \in \mathcal{X}$ and every sequence $\left\{ x_n \right\}_{n \in \mathbb{N}}$ converging to $x$,
\begin{equation*}
\liminf_{n \rightarrow \infty} F_n(x_n) \geq F(x),
\end{equation*}
\item  \textbf{Limsup inequality:} For every $x \in \mathcal{X}$ there exists a sequence $\left\{ x_n \right\}_{n \in \N}$ converging to $x$ satisfying
\begin{equation*}
\limsup_{n \rightarrow \infty} F_n(x_n) \leq F(x).
\end{equation*}
\item \textbf{Compactness:} Every bounded sequence $\{ x_n \}_{n \in \N}$ satisfying
\[  \sup_{n \in \N} F_n(x_n) < \infty \]
is precompact. 
\end{itemize}
We say that $F$ is the $\Gamma$-limit of the sequence of functionals $\left\{F_n \right\}_{n \in \N}$ (with respect to the metric $d_\mathcal{X}$). 
\label{defGamma}
\end{definition} 

\begin{remark}
A sequence $\{ x_n\}_{n \in \N} \subseteq \mathcal{X}$ like the one appearing in the limsup inequality is said to be a recovery sequence for $x$.  It is straightforward to show that if one can find recovery sequences for all elements in a set $\mathcal{X}'$ satisfying:
\begin{enumerate}
\item For all $x \in \mathcal{X}$ there exists a sequence $\{ x_l \}_{l\in \N} \subseteq \mathcal{X}'$ such that $x_l \rightarrow x$ and $F(x_l) \rightarrow F(x)$, 
\end{enumerate}
then one can find recovery sequences for all elements in $\mathcal{X}$; a set $\mathcal{X}'$ satisfying (1) above is said to be dense in $\mathcal{X}$ with respect to $F$. This fact follows from a  simple diagonal argument (see \cite{DalMaso}).
\label{density} 
\end{remark}

The most relevant property of $\Gamma$-convergence (in particular for our purposes) is presented in the following proposition which can be found in \cite{DalMaso}.
\begin{proposition}
	\label{porpGamma}
Let $(\mathcal{X}, d_{\mathcal{X}})$ be a metric space and let  $F_n: \mathcal{X} \rightarrow[0,\infty] $ be a sequence of functionals that are not identically equal to $\infty$. If $\{ F_n \}_{n \in\N}$ $\Gamma$-converges towards $F$ then, 
\begin{enumerate}
	\item Any sequence $\{ x_n^*\}_{n \in \N}$ where $x_n^*$ is a minimizer of $F_n$ is precompact, and each of its accumulation points is a minimizer of $F$. In particular, if $F$ has a unique minimizer, then $\{ x_n^* \}_{n \in \N}$ converges towards it.	
	\item We have
	\[\lim_{n \rightarrow \infty} \min_{x \in \X} F_n(x) = \min_{x \in \X} F(x).  \]
\end{enumerate}
\end{proposition}

\nc

We have presented the notion of $\Gamma$ convergence in the above generality because some of the $\Gamma$-limits we will consider in this paper are taken in the context of the metric space $L^1(D)$ 
whereas others are taken in the context of the metric space $TL^1(D)$. Also, notice that when the functionals are allowed to be random (as it is the case for the graph total variation), 
$\Gamma$-convergence has to be interpreted as in the Definition 2.11 in \cite{GarciaTrillos2015}. 

To conclude this section we list two important properties of the weighted total variation functional $TV( \cdot; h)$. 
Let $h : D \rightarrow \R$ be a continuous function which is bounded above and below by positive constants. The first property is a representation formula for $TV(u;h)$ in terms of the distributional derivative of $u$.
Indeed, it follows from the work in \cite{Baldi} that for every $u \in BV(D)$ one can write
\begin{equation}
 TV(u; h) = \int_{D} h(x) d \lvert Du \rvert(x),
\label{repformula}
\end{equation}
where in the above $Du$ stands for the distributional derivative of $u$ (which in general is a signed measure) and $\lvert Du \rvert $ stands for the total variation measure associated to $Du$. 

The second property that is relevant for our purposes is the \textit{coarea formula} which states that for every $u \in BV(D)$,
\begin{equation}
 TV(u;h) = \int_{-\infty}^{\infty} TV(\mathds{1}_{\{ x  \: : \:  u(x) >t \}} ; h) dt.
\label{coarea}
\end{equation}
This formula says that the total variation of $u$ can be written in terms of the total variation of its level sets.  See Theorem 13.25 in \cite{Leoni}.

\section{Proofs of Theorems}
\label{Proofs}

The main auxiliary result that we need to establish Theorem \ref{Main} is the following.

\begin{proposition}

For every $n \in \N$ let  $\veps_n, \hat{\veps}_n: D \rightarrow (0,\infty)$ be functions satisfying 
\[ \sup_{x \in D} \left \lvert  \frac{\veps_n(x)}{\hat{\veps}_n(x)} -1  \right \rvert \rightarrow   0 \quad, \text{ as } n \rightarrow \infty,   \]
and
\[  \sup_{x \in D}   \veps_n(x)    \rightarrow 0 , \text{ as } n \rightarrow \infty.  \]
For $n\in \N$ consider the energy $F_n : L^1(D) \rightarrow [0,\infty)$ defined by
\[ F_n(u) : = \int_{D} \hat{f}_n(x)\left( \int_{D} \eta_{\hat{\veps}_n(x) }(x-y) \lvert  u(x) - u(y) \rvert \rho(y) dy\right) \rho(x) dx, \quad u \in L^1(\nu),  \]
where
\[ \hat{f}_n(x) :=  \frac{(\hat{\veps}_n(x))^{d}}{  ( \nu( B(x, \veps_n(x)) )     )^{1+ 1/d}  }.\]
Then, as $n \rightarrow \infty$, $F_n$ $\Gamma$-converges in the $L^1(D)$ sense towards $F$, where $F$ is given by
\[ F(u):=   \frac{\sigma_\eta}{\alpha_d^{1+1/d}} TV(u; \rho^{1-1/d}). \]
\label{Lemma1}
\end{proposition}
Proposition \ref{Lemma1} is a consequence of the next lemma.
\begin{lemma}
Let $\rho_1,\rho_2: \R^d \rightarrow (0,\infty)$ be two Lipschitz continuous functions which are bounded above and below by positive constants in $D$. Let $\mu$ be the Borel measure on $\R^d$ given by $d \mu(x) = \rho_2(x) dx$. 
Let  $\veps_n, \hat{\veps}_n: D \rightarrow (0,\infty)$ be functions satisfying 
\[ \sup_{x \in D} \left \lvert  \frac{\veps_n(x)}{\hat{\veps}_n(x)} -1  \right \rvert \rightarrow   0 \quad, \text{ as } n \rightarrow \infty,   \]
and
\[  \sup_{x \in D}   \veps_n(x)    \rightarrow 0 , \text{ as } n \rightarrow \infty.  \]
Let $F_n : L^1(D) \rightarrow [0,\infty)$ be defined by
\[ F_n(u) : = \int_{D} \hat{f}_n(x) \left( \int_{D} \eta_{\hat\veps_n(x) }(x-y) \lvert  u(x) - u(y) \rvert \rho_1(y) dy\right) \rho_1(x) dx, \quad u \in L^1(\nu),  \]
where
\[ \hat{f}_n(x) :=  \frac{(\hat{\veps}_n(x))^{d}}{  ( \mu( B(x, \veps_n(x)) )     )^{1+ 1/d}  }.\]
Then, as $n \rightarrow \infty$, $F_n$ $\Gamma$-converges in the $L^1(D)$ sense towards 
\[ F(u):=   \frac{\sigma_\eta}{\alpha_d^{1+1/d}} TV\left(u; \frac{\rho_1^2}{\rho_2 ^{1+1/d} }\right). \]
\label{LemmaAux}
\end{lemma}
\begin{proof}

Let us start by proving the liminf inequality. That is, let us prove that for every $u \in L^1(D)$ and every sequence $\{ u_n \}_{n \in \N}\subseteq L^1(D)$ satisfying $u_n \converges{L^1(D)} u$, we have
\begin{equation} 
\liminf_{n \rightarrow \infty} F_n(u_n) \geq F(u). 
\label{LiminfineqG}
\end{equation}

The following simplifications are standard. First, working along subsequences we may assume without the loss of generality that the liminf is actually a limit. In addition, we may assume that both of the terms involved in
inequality \ref{LiminfineqG} are finite. We now split the proof of \eqref{LiminfineqG} into several steps.

\textbf{Step 1:} Instead of working with the energy $F_n$ directly, we first consider a simpler related energy $E_n$ defined by
\[ E_n(v):=\int_{D} \left(  \frac{1}{\hat{\veps}_n(x)}  \int_{D}  \eta_{\hat\veps_n(x)}(\lvert x-y  \rvert) \lvert v(x) - v(y) \rvert  dy\right) g(x) dx   , \quad v \in L^1(D),\] 
where 
\[ g(x):= \frac{(\rho_1(x))^2}{(\alpha_d\rho_2(x))^{1+1/d}},  \quad x \in D.\]
Notice that for every $x \in D$ we have
\[ \left \lvert \mu(B(x, \hat{\veps}_n(x)))  - \alpha_d \rho_2(x) (\hat{\veps}_n(x)) ^d \right \rvert \leq \alpha_d \Lip(\rho_2) (\hat{\veps}_n(x))^{d+1}.    \] 
From the previous inequality, the Lipschitz continuity of $\rho_1, \rho_2$ and the assumptions on $\hat{\veps}_n(\cdot)$ and $\veps_n(\cdot)$, we deduce that for every $x \in D$ and $y \in B(x, \hat{\veps}_n(x))$, 
\[ \left \lvert    \frac{g(x)}{\hat{\veps}_n(x)} - \hat{f}_n(x)\rho_1(x)\rho_1(y)  \right \rvert \leq C \sup_{z \in D} \veps_n(z),   \]
where $C$ is a constant that does not depend on $x$ or $y$. It then follows that
\[  \lim_{n \rightarrow \infty}  \lvert F_n(u_n) -  E_n(u_n)  \rvert =0 .\]
In particular, to obtain \eqref{LiminfineqG} we may as well show that
\[  \liminf_{n \rightarrow \infty} E_n(u_n) \geq F(u).  \]


\textbf{Step 2:} 
Let $B$ be a closed ball contained in $D$. We claim that 
\[  \liminf_{n \rightarrow \infty}  \int_{B} \left(  \frac{1}{\hat{\veps}_n(x)}  \int_{B}  \eta_{\hat\veps_n(x)}(\lvert x-y  \rvert) \lvert u_n(x) - u_n(y) \rvert  dy\right)  dx \geq \sigma_\eta TV_B(u), \]
where $TV_B(u)$ is defined as
\begin{equation} 
TV_B(u) :=  \sup \left\{    \int_{B^\circ} u(x) \divergence(\zeta)dx \: : \: \zeta \in C^\infty_c(B^\circ: \R^d),  \quad \lvert \zeta(x) \rvert \leq 1 \quad \forall x \in B^\circ    \right\},
\label{TVB}
\end{equation}
where in the above $B^\circ $ stands for the interior of the closed ball $B$ (the ball without its boundary).

The claim follows from Theorem 8  in \cite{Ponce}. The only difference in the setting we consider here and the setting considered in \cite{Ponce} 
is that in our case the length-scale $\hat{\veps}_n$ depends on location $\hat{\veps}_n:= \hat{\veps}_n(x)$. The fact 
that $\hat{\veps}_n(x)$ converges to zero uniformly over $x \in D$ as $n \rightarrow \infty$ is enough to make the proof in \cite{Ponce} carry through with essentially no modifications. 
We remark that the arguments in \cite{Ponce} were also used in \cite{GarciaTrillos2015}; in \cite{GarciaTrillos2015} the presence of a non-constant density makes computations a bit more tedious.

\textbf{Step 3:} Suppose that $u \in BV(D)$ is of the form $u = \mathds{1}_A$ for some measurable set $A\subseteq D$. That is, suppose that $u$ is the indicator function of a set with finite perimeter (with respect to $D$).
In the Appendix we show the following  fact: there exists a sequence $\{ \mathcal{F}_l \}_{l \in \N}$ of collections of closed balls contained in $D$ satisfying:
\begin{enumerate}
\item For every $l \in \N$, the family $\mathcal{F}_l$ is finite.
\item For every $l \in \N$, the balls $B:=\overline{B(x_B, r_B)}$ in $\mathcal{F}_l$ are pairwise disjoint. 
\item $ \lim_{l \rightarrow \infty}  \max \{ r_B \: : \: B \in \mathcal{F}_l \}    =0 $.
\item For every $l \in \N$, and for every $B \in \mathcal{F}_l$,  $\lvert  Du \rvert(\partial B) =0$.
\item $ \lim_{l\rightarrow \infty}   \lvert D u \rvert \left( \bigcup_{B \in \mathcal{F}_l} B \right)  =  \lvert D u \rvert(D)$.
\end{enumerate} 
With the families $\{ \mathcal{F}_l \}_{l \in \N}$ at hand, we may now consider the functions  
\[  g_l(x) := \begin{cases}  \min_{y \in B} g(y) , \text{ if } \quad x \in B, \quad B \in \mathcal{F}_l  \\ 0 , \quad \text{otherwise}.    \end{cases}  \]   
Then, for every $l \in \N$
\begin{align*}
\begin{split}
   \liminf_{n \rightarrow \infty  }  E_n(u_n) & \geq \liminf_{n \rightarrow \infty} \sum_{B \in \mathcal{F}_l}  \int_{B} \left(  \frac{1}{\hat{\veps}_n(x) }  \int_{B} \eta_{\hat{\veps}_n(x)}( \lvert x-y\rvert) \lvert  u_n(x) - u_n(y) \rvert   dy  \right)  g(x)  dx  
\\& \geq \sum_{B \in \mathcal{F}_l} \liminf_{n \rightarrow \infty}  \int_{B} \left(  \frac{1}{\hat{\veps}_n(x) }  \int_{B} \eta_{\hat{\veps}_n(x)}( \lvert x-y\rvert) \lvert  u_n(x) - u_n(y) \rvert   dy   \right)  g(x)  dx 
\\& \geq \sum_{B \in \mathcal{F}_l} g_l(x_B) \liminf_{n \rightarrow \infty}  \int_{B} \left(  \frac{1}{\hat{\veps}_n(x) }  \int_{B} \eta_{\hat{\veps}_n(x)}( \lvert x-y\rvert) \lvert  u_n(x) - u_n(y) \rvert   dy \right)    dx
\\&  \geq \sum_{B \in \mathcal{F}_l} g_l(x_B) \sigma_\eta  TV_B(u)  
\\& = \sum_{B \in \mathcal{F}_l} g_l(x_B) \sigma_\eta  \int_{B} d\lvert  Du \rvert(x)
\\&=\sigma_\eta \int_{D} g_l(x) d \lvert  D u\rvert(x),
\end{split}
\end{align*}
where in the fourth inequality we used Step 2 and in the first equality we used Lemma 15.12 in \cite{Maggi} combined with property (4) for the balls in $\mathcal{F}_l$. 

Finally, from properties (3) and (5) of $\{\mathcal{F}_l \}_{l \in \N}$ and from the Lipschitz continuity of $g$ we know that 
\[ \lim_{l \rightarrow \infty} g_l(x)  = g(x) , \quad  \text{ for }  \lvert Du \rvert \text{- a.e. }   x \in D. \]
The dominated convergence theorem implies that 
\[ \lim_{l \rightarrow \infty} \int_{D} g_l(x) d \lvert  D u\rvert(x) = \int_{D} g(x)  d \lvert  D u\rvert(x)  =  TV\left( u ; g \right), \]
where in the last equality we have used the representation formula \eqref{repformula}. We conclude that for functions $u=\mathds{1}_A\in BV(D)$ (i.e. indicator functions for sets of finite perimeter) inequality \eqref{LiminfineqG} holds.

\textbf{Step 4:} In order to prove \eqref{LiminfineqG} for general $u \in BV(D)$ we use Step 3 and the coarea formula for the energies $E_n$ and $F$. Notice that the coarea formula for $F$ is given in \eqref{coarea}, whereas the coarea formula for $E_n$ follows directly from the identity
\[ \lvert a-b \rvert  = \int_{-\infty}^\infty \lvert \mathds{1}_{a> t} - \mathds{1}_{b>t}  \rvert  dt , \quad \forall a , b \in \R.\]

If $u_n \converges{L^1(D)} u $, then for 
Lebesgue-a.e. $t \in \R$ we have 
\[  \mathds{1}_{\{ u_n > t \}}  \converges{L^1(D)}  \mathds{1}_{\{ u >t  \}}, \quad \text{as } n \rightarrow \infty.\] 
We may then apply Step 3 to conclude that for almost every $t \in \R$
\[  \liminf_{n \rightarrow  \infty} E_n( \mathds{1}_{\{ u_n > t \}} ) \geq F(\mathds{1}_{\{ u > t \}}).\]
Integrating the above inequality with respect to $t$, and using Fatou's lemma and the coarea formulas for the energies $E_n$ and $F$, we deduce that
\[  \liminf_{n \rightarrow \infty}E_n(u_n) = \liminf_{n \rightarrow \infty} \int_{-\infty}^\infty  E_n( \mathds{1}_{\{ u_n > t \}} ) dt \geq  \int_{-\infty}^\infty   \liminf_{n \rightarrow \infty} E_n( \mathds{1}_{\{ u_n > t \}} ) dt    \geq  \int_{-\infty}^\infty    E( \mathds{1}_{\{ u> t \}} ) dt   = F(u).\]

Having established the liminf inequality we now proceed to establish the limsup inequality. From remark \ref{density} and the density of $C_c^\infty(\R^d) $ functions in $L^1(D)$ with respect to $TV$ (see Proposition 2.4 in \cite{GarciaTrillos2015}), it is enough to find a recovery sequence for every $u$ obtained as the restriction to $D$ of a function in $ C^\infty_c(\R^d) $.  We actually show that for every $u \in C_c^\infty(\R^d)$
\begin{equation}
 \limsup_{n \rightarrow \infty} F_n(u)  \leq F(u). 
 \label{fixedu}
\end{equation}
In other words $u_n:= u$ for all $n$, is a recovery sequence for $u$. To prove this, notice that 
\begin{align*}
E_{n}(u) &=  \int_{D} \left( \frac{1}{\hat \veps_n(x)}\int_{ B(x, \hat\veps_n(x)) \cap D}\eta_{\hat\veps_n(x)}(x-y) |u(x)-u(y)|dy \right) g(x)dx
\\& \leq \int_{D} \left( \frac{1}{\hat \veps_n(x)}\int_{ B(x, \hat\veps_n(x)) }\eta_{\hat\veps_n(x)}(x-y) |u(x)-u(y)|dy \right) g(x)dx
\\& \leq \int_{0}^{1}  \int_{D} \left( \frac{1}{\hat\veps_n(x)} \int_{B(x,  \hat\veps)} \eta_{\hat\veps_n(x)}(x-y) |\nabla u(x+t(y-x)) \cdot (x-y)| dy \right)g(x)dx dt,
\end{align*}
which follows from the fundamental theorem of calculus. Now, for every fixed $x \in D$, we use the change of variables $h:= \frac{y-x}{\hat \veps_n(x)}$ in the above expression 
to rewrite \[  \frac{1}{\hat\veps_n(x)} \int_{B(x,  \hat\veps(x))} \eta_{\hat\veps_n(x)}(x-y) |\nabla u(x+t(y-x)) \cdot (x-y)| dy  = \int_{B(0, 1)} \eta(h) |\nabla u(x+t\hat{\veps}_n(x) h ) \cdot h | dh.    \]  
But because $u \in C_c^\infty(\R^d)$ and $\sup_{x \in D} \hat \veps_n(x) \rightarrow 0$ as $n \rightarrow \infty$, it follows that
\[ \lim_{n \rightarrow \infty} \sup_{x \in D} \left \lvert  \int_{B(0, 1)} \eta(h) |\nabla u(x+t\hat{\veps}_n(x) h ) \cdot h | dh -  \int_{B(0, 1)} \eta(h) |\nabla u(x ) \cdot h | dh  \right \rvert =0. \]
Hence, 
\[ \limsup_{n \rightarrow \infty} E_n(u) \leq  \int_{0}^{1}  \int_{D} \left( \int_{B(0,1)} \eta(h) |\nabla u(x) \cdot h| dh \right)g(x)dx dt  = \sigma_\eta TV(u; g). \]
\end{proof}

With the previous lemma at hand we can now establish Proposition \ref{Lemma1}.

\begin{proof}[Proof of Proposition \ref{Lemma1}]
%

We make use of Lemma \ref{LemmaAux} and for that purpose we first need to approximate the function $\rho: D \rightarrow \R$ from above and from below using appropriate Lispchitz functions.  
More precisely, for every $s \in \N$ let
 \[ \tilde{\rho}_{1,s}(x):=   \inf_{y \in D}  \left\{ \rho(y) + s |x-y| \right\}  , \quad \rho_{2,s}(x) := \sup_{y \in D}  \left\{ \rho(y) - s |x-y| \right\}.
  \]
Notice that the functions $\tilde{\rho}_{1,s}, \rho_{2,s}$  are Lipschitz continuous, and as $s \rightarrow \infty$,
\[  \rho_{1,s}(x) \nearrow \rho(x) , \quad \rho_{2,s}(x) \searrow \rho(x), \quad \forall x \in D. \]
For every $s \in\N$ we modify the function $\tilde{\rho}_{1,s}$ slightly so as to create a Lipschitz function $\rho_{1,s}$ that in addition to minorizing $\rho$, is such that for all large enough $n \in \N$ 
\begin{equation} 
\int_{B(x, \veps_n(x))} \rho_{1,s}(y) dy \leq \int_{B(x, \veps_n(x)) \cap D} \rho(y)dy = \nu(B(x, \veps_n(x))), \quad \forall x \in D . 
\label{BoundBelow}
\end{equation}
To achieve this, we use the regularity assumption on the boundary of $D$ as follows. From the fact that $D$ is an open and bounded set with Lipschitz boundary it follows (see \cite{Gris}, Theorem 1.2.2.2) that there exists a cone $\mathcal{C} \subseteq \R^d$ with non-empty interior and vertex at the origin, a family of rotations $\left\{ R_{x} \right\}_{x \in D}$ and a number $1>\zeta>0$ such that for every $x \in D,$
\[  x + R_x(\mathcal{C} \cap B(0,\zeta))  \subseteq D. \] 

For every $s\in \N$, let $\xi_s : \R^d \rightarrow [0,1]$ be a smooth cutoff function satisfying $\xi_s(x)\equiv 1$ in $\{y \in D \: : \: \dist(y, \partial D )>1/s \}$ and $\xi \equiv 0$ in 
$\{y \in D \: : \: \dist(y, \partial D ) <1/2s \}$, and let us define 
\[ \rho_{1,s}(x):= \frac{\lvert B(0,1) \cap \mathcal{C} \rvert}{\lvert B(0,1) \rvert}  \rho_{min} (1 -\xi_s(x) )   +  \tilde{\rho}_{1,s}(x) \xi_s(x), \quad x \in \R^d. \]  

It is clear that the functions $\rho_{1,s}$ are Lispchitz, they minorize $\rho$ and they converge to $\rho$ pointwise. 
In addition, for any fixed $s$, if $n \in \N$ is large enough so that in particular $\sup_{x \in D} \veps_n(x) <\min\{ \zeta, \frac{1}{8s} \}$, then \eqref{BoundBelow} holds. Indeed, notice that if $x \in D$ is $1/{4s}$ units 
away from $\partial D$ then \eqref{BoundBelow} follows directly form the fact that $\rho_{1,s} \leq \rho$; on the other hand, if $x \in D$ is within distance $1/{4s}$ from $\partial D$, we conclude that 
\begin{align*}
\begin{split}
\int_{B(x, \hat{\veps}_n(x))} \rho_{1,s}(y) dy = \lvert \mathcal{C} \cap B(0,1) \rvert  \rho_{min} (\hat{\veps}_n(x))^d  \leq \int_{(x+ R_x(\mathcal{C}))\cap B(x, \hat{\veps}_n(x)) } \rho(y)dy 
\\ \leq \int_{B(x, \hat{\veps}_n(x)) \cap D} \rho(y) dy =\nu( B(x, \hat{\veps}_n(x))).
\end{split}
\end{align*}

Let us now establish the limsup inequality. We introduce the functionals $F_{n,s}: L^1(D) \rightarrow [0, \infty)$ given by
\[ F_{n,s}(u):=   \int_{D} \tilde{f}_{n,s}(x) \left( \int_{D} \eta_{\veps_n(x) }(x-y) \lvert  u(x) - u(y) \rvert \rho_{2,s}(y) dy\right) \rho_{2,s}(x) dx, \quad u \in L^1(D),  \]
where
\[ \tilde{f}_{n,s}(x) :=  \frac{(\tilde{\veps}_n(x))^{d}}{  ( \mu_s( B(x, \veps_n(x)) )     )^{1+ 1/d}  }, \]
and
\[ d \mu_s(x):= \rho_{1,s}(x) dx.\]
It follows that for every fixed $s \in \N$,  for all large $n \in \N$ 
\[ F_n(v) \leq F_{n,s}(v) , \quad \forall v \in L^1(D). \]
Hence, from \eqref{fixedu} it follows that for every arbitrary $u \in C^\infty_c(\R^d)$ 
\[ \limsup_{n \rightarrow \infty} F_{n}(u) \leq \limsup_{n \rightarrow \infty} F_{n,s}(u) \leq \frac{\sigma_\eta}{\alpha_d^{1+1/d}} \int_{D} \frac{\rho_{2,s}^2(x)}{\rho_{1,s} ^{1+ 1/d}(x) } d \lvert Du \rvert(x).\]
Taking $s \rightarrow \infty$ in the above inequality we deduce that
\[  \limsup_{n \rightarrow \infty} F_{n}(u) \leq \frac{\sigma_\eta}{\alpha_d^{1+1/d}} TV(u; \rho^{1-1/d} ).  \]
Remark \eqref{density} and the density of $C_c^\infty(\R^d)$ functions with respect to $TV$ imply the desired result.


To establish the liminf inequality, it is enough to change the roles of $\rho_{1,s}$ and $\rho_{2,s}$ in the definition of the functional $F_{n,s}$ and use Lemma \ref{LemmaAux} to conclude that if $u_n \converges{L^1(D)} u $, 
then for every $s \in \N$
\[ \liminf_{n \rightarrow \infty} F_{n}(u_n) \geq \liminf_{n \rightarrow \infty} F_{n,s}(u_n) \geq \frac{\sigma_\eta}{\alpha_d^{1+1/d}} \int_{D} \frac{\rho_{1,s}^2(x)}{\rho_{2,s} ^{1+ 1/d}(x) } d \lvert Du \rvert(x).\]
Taking the limit as $s \rightarrow \infty$ on the right hand side of the above expression we obtain the desired result.
 
\end{proof}

We are ready to establish Theorem \ref{Main}.

\begin{proof}[Proof of liminf inequality in Theorem \ref{Main}:]
Let us consider  the cone $\mathcal{C}$, the family of rotations $\left\{ R_{x} \right\}_{x \in D}$, and the number $1>\zeta>0$ 
introduced in the proof of Proposition \ref{Lemma1}. For $0\leq \gamma_1< \gamma_2 < \zeta$ and $x \in D$ let $A(x;\gamma_1, \gamma_2)$ be the annulus $A(x; \gamma_1, \gamma_2):= B(x, \gamma_2)\setminus B(x , \gamma_1) $. Then, there exist constants $\mathcal{K}_1$ (only depending on $d$, and on the cone $\mathcal{C}$) and $\mathcal{K}_2$ (only depending on $d$), such that for every $0\leq \gamma_1< \gamma_2 < \zeta$,
\begin{equation}
 \mathcal{K}_1 \rho_{min}  (\gamma_2)^{d-1} \cdot (\gamma_2- \gamma_1   )   \leq  \nu\left( A(x; \gamma_1, \gamma_2) \cap (x+ R_x(\mathcal{C})) \right)   \leq  \nu\left( A(x; \gamma_1, \gamma_2)  \right) , \quad  \forall x \in D,
 \label{lowerbound}
 \end{equation}
and
\begin{equation}
 \nu\left( A(x; \gamma_1, \gamma_2)  \right) \leq \mathcal{K}_2 \rho_{max}  (\gamma_2)^{d-1} \cdot (\gamma_2- \gamma_1 ), \quad \forall x \in D. 
 \label{upperbound}
 \end{equation}

We work in a set with full probability for which the maps $\left\{ T_n \right\}_{n \in \N}$ from \eqref{Tn} exist. 
For every $x \in D$ we let $\veps_n(x)$ be the number for which
\[ \nu\left(  B( x, \veps_n(x)) \right)  = \bar{\veps}_n^d = \frac{k_n}{n}. \] 
On the other hand, we define $\dot{\veps}_n(x)$ and $\hat{\veps}_n(x)$ to be 
\[  \dot{\veps}_n(x) := \veps_n(x) - \frac{ \mathcal{K}_2\rho_{max}\delta_n}{\mathcal{K}_1\rho_{min}}, \]
\[ \hat{\veps}_n(x) :=  \dot{\veps}_n(x) -3\delta_n,\]
where $\delta_n$ is as in \eqref{Tn}. 

Notice that the assumption $n \gg k_n \gg \log(n)$ is equivalent to $1 \gg \overline{\veps}_n \gg \delta_n  $ where we recall $\overline{\veps}_n$ is defined in \eqref{hateps}. Combining this with \eqref{lowerbound} we deduce that

\begin{equation}
 \lim_{n \rightarrow \infty} \sup_{x\in D} \veps_n(x)  =0,  \quad  \lim_{n \rightarrow \infty}  \sup_{x \in D} \left \lvert\frac{\veps_n(x)}{\hat{\veps}_n(x)} -1 \right\rvert=0, \quad  \lim_{n \rightarrow \infty} \frac{\delta_n}{\inf_{x \in D} \hat{\veps}_n(x)}=0.   
 \label{Propshatveps}
 \end{equation}

Now, from \eqref{lowerbound} we see that for all large enough $n$, and for all $x \in D$,
\begin{align}
\begin{split}
 \nu(B(x, \dot{\veps}_n(x)) )  & =  \nu(B(x, \veps_n(x)) )  - \nu(A(x; \dot{\veps}_n(x), \veps_n(x))) 
\\&\leq \frac{k_n}{n} - \mathcal{K}_2 \rho_{max}  (\veps_n(x))^{d-1} \delta_n.  
\end{split}
\end{align}
Thus, 
\begin{align*}
\begin{split}
\nu_n(  B(x, \dot{\veps}_n(x) )  ) & =  \nu_n(  B(x, \dot{\veps}_n(x) )  )  - \nu(  B(x, \dot{\veps}_n(x) )  ) + \nu(  B(x, \dot{\veps}_n(x) )  ) 
\\&  \leq \nu(  B(x, \dot{\veps}_n(x) + \lVert Id - T_n \rVert_\infty )  )  - \nu(  B(x, \dot{\veps}_n(x) )  ) + \nu(  B(x, \dot{\veps}_n(x) )  )
\\ & \leq \mathcal{K}_2 \rho_{max}\veps_n(x)^{d-1} \lVert Id - T_n \rVert_\infty +  \frac{k_n}{n} -  \mathcal{K}_2 \rho_{max} (\veps_n(x))^{d-1} \cdot \delta_n
\\&\leq\mathcal{K}_2 \rho_{max}\veps_n(x)^{d-1} \delta_n +  \frac{k_n}{n} -  \mathcal{K}_2 \rho_{max} {\veps}_n(x)^{d-1} \delta_n
\\& = \frac{k_n}{n},
\end{split}
\end{align*}
where in the first inequality we have used the fact that $T_n$ is a transportation map between $\nu$ and $\nu_n$, and 
in the second inequality we have used \eqref{upperbound}.  Combining the previous inequality with the fact that $B(T_n(x) , \dot{\veps}_n(x) - \delta_n) \subseteq B(x, \dot{\veps}_n(x))$ we deduce that 
\[ \nu_n(B(T_n(x) , \hat{\veps}_n(x) -\delta_n )) \leq \frac{k_n}{n}.\]
In particular, if $y \in D$ is such that $T_n(y) \not \sim_{k_n} T_n(x)$, then from Remark \ref{Remknn} it follows that 
\[  \hat{\veps}_n(x) = \dot{\veps}_n(x)  - 3 \delta_n <  \lvert T_n(x) - T_n(y) \rvert  -2 \delta_n \leq  \lvert T_n(x) - T_n(y) \rvert- \lvert x- T_n(x)\rvert - \lvert y-T_n(y) \rvert \leq \lvert x-y \rvert.\]
We conclude that for all large enough $n$,  $J_{k_n}(T_n(x) , T_n(y)) =0 $ implies $\hat{\veps}_n(x) < \lvert x-y \rvert$. 
In particular, we deduce that for all large enough $n \in \N$,
\begin{equation}
J_{k_n}(T_n(x), T_n(y))  \geq \eta \left( \frac{\lvert  x-y  \rvert}{ \hat{\veps}_n(x)} \right), \quad \forall x,y \in D.
\label{RelJandEta}
\end{equation}
Summarizing, by contracting the radius $\veps_n(x)$ a little bit we obtained $\hat{\veps}_n(x)$ satisfying \eqref{Propshatveps} and \eqref{RelJandEta}.

%
%
%

Let us now consider a sequence of functions $\left\{ u_n \right\}_{n \in \N}$ with $u_n \in L^1(\nu_n)$ for which $u_n \converges{TL^1} u$ for some $u \in L^1(\nu)$.   
It follows from \eqref{RelJandEta} that for all large enough $n \in \N$,
\begin{align}
\begin{split}
GTV_{n, k_n}(u_n) &= \frac{1}{n^2 \overline{\veps}_n^{d+1}} \sum_{i,j} J_{k_n}(\x_i, \x_j) \lvert u_n(\x_i) - u_n (\x_j)\rvert 
\\ &=  \frac{1}{n^2 }  \sum_{i=1}^n \frac{1}{\overline{\veps}_n^{d+1}} \left( \sum_{j=1}^n J_{k_n}(\x_i, \x_j) \lvert u_n(\x_i) - u_n (\x_j)\rvert \right) 
\\&= \int_{D} \frac{1}{\overline{\veps}_n^{d+1}} \left( \int_{D} J_{k_n}(T_n(x), T_n(y)) \lvert u_n\circ T_n(x) - u_n\circ T_n(y)\rvert  d \nu(y) \right) d \nu(x)
\\&\geq  \int_{D} \frac{1}{\overline{\veps}_n^{d+1}} \left( \int_{D} \eta\left( \frac{\lvert x-y \rvert}{\hat{\veps}_n(x)} \right)  \lvert u_n\circ T_n(x) - u_n\circ T_n(y)\rvert  d \nu(y) \right) d \nu(x)
\\& =  \int_{D} \hat{f}_n(x)\left( \int_{D} \eta_{\hat{\veps}_n(x)}\left( \lvert x-y \rvert  \right)  \lvert u_n\circ T_n(x) - u_n\circ T_n(y)\rvert  d \nu(y) \right) d \nu(x) 
\\&= F_n(u_n\circ T_n),
\end{split}
\label{ineqaux0}
\end{align}
where for an arbitrary function $v \in L^1(\nu)$, $F_n(v)$ is defined as
\[ F_n(v):= \int_{D} \hat{f}_n(x) \left( \int_{D} \eta_{\hat{\veps}_n(x)}\left( \lvert x-y \rvert  \right)  \lvert v(x) - v(y)\rvert  d \nu(y) \right) d \nu(x), \]
and the function $\hat{f}_n$ is given by $ \hat{f}_n(x):= \frac{\hat{\veps}_n(x)^d}{( \nu(B(x, \veps_n(x))) )^{1+1/d}} $.
Since $u_n \converges{TL^1} u$ implies that $u_n \circ T_n \converges{L^1(\nu)} u$, it follows from Proposition \eqref{Lemma1} and \eqref{ineqaux0} that
\[   \liminf_{n \rightarrow \infty} GTV_{n, k_n}(u_n)  \geq \liminf_{n \rightarrow \infty} F_n(u_n\circ T_n)  \geq \frac{\sigma_\eta}{\alpha_d^{1+1/d}} TV(u; \rho^{1-1/d}),\]
which establishes the liminf inequality in Theorem \ref{Main}.
\end{proof}

\begin{proof}[Proof of the limsup inequality in Theorem \eqref{Main}:]
We work in a set with full probability for which the maps $\left\{ T_n \right\}_{n \in \N}$ from \eqref{Tn} exist. By an analogue construction to the one used in the proof of the liminf inequality (this time enlarging $\veps_n(x)$ instead of contracting it), 
we may construct a function $\hat{\veps}_n$ satisfying \eqref{Propshatveps} and 
\[  J_{k_n}(T_n(x) , T_n(y)) \leq \eta \left( \frac{\lvert x-y \rvert}{\hat{\veps}_n(x)} \right), \quad \forall x,y \in D. \]

Let $u \in C_c^\infty(\R^d)$ and for every $n \in \N$ let $u_n$ be the function in $L^1(\nu_n)$ defined by
\[ u_n(\x_i):= u(\x_i), \quad i=1, \dots, n.\] 
In other words, $u_n$ is the function $u$ restricted to the point cloud $X_n$. Then,

\begin{align}
\begin{split}
GTV_{n, k_n}(u_n) &= \frac{1}{n^2 \overline{\veps}_n^{d+1}} \sum_{i,j} J(\x_i, \x_j) \lvert u(\x_i) - u(\x_j)\rvert 
\\&= \int_{D} \frac{1}{\overline{\veps}_n^{d+1}} \left( \int_{D} J(T_n(x), T_n(y)) \lvert u\circ T_n(x) - u\circ T_n(y)\rvert  d \nu(y) \right) d \nu(x)
\\& \leq  \int_{D} \frac{1}{\overline{\veps}_n^{d+1}} \left( \int_{D} \eta\left( \frac{\lvert x-y \rvert}{\hat{\veps}_n(x)} \right)  \lvert u_n\circ T_n(x) - u_n\circ T_n(y)\rvert  d \nu(y) \right) d \nu(x)
\\& \leq  \int_{D} \hat{f}_n(x) \left( \int_{D} \eta_{\hat{\veps}_n(x)}\left( \lvert x-y \rvert  \right)  \lvert u(x) - u(y)\rvert  d \nu(y) \right) d \nu(x)  + \mathcal{I}_n + \mathcal{II}_{n}
\\&= F_n(u)  +   \mathcal{I}_n + \mathcal{II}_{n},
\end{split}
\label{ineqaux1}
\end{align}
where 
\[ \mathcal{I}_n:=  \int_{D} \hat{f}_n(x) \left( \int_{D} \eta_{\hat{\veps}_n(x)}\left( \lvert x-y \rvert  \right)  \lvert u(x) - u(T_n(x))\rvert  d \nu(y) \right) d \nu(x),   \]
and
\[ \mathcal{II}_n:=  \int_{D} \hat{f}_n(x) \left( \int_{D} \eta_{\hat{\veps}_n(x)}\left( \lvert x-y \rvert  \right)  \lvert u(y) - u(T_n(y))\rvert  d \nu(y) \right) d \nu(x),   \]
But notice that 
\[ \mathcal{I}_n , \mathcal{II}_{n} \leq  \frac{ C \lVert \nabla u \rVert_{\infty} \lVert Id - T_n \rVert_{\infty}}{\inf_{x \in D} \hat{\veps}_n(x)  },  \]
for some constant $C$. Hence, $\lim_{n \rightarrow \infty} \mathcal{I}_n  = \lim_{n \rightarrow \infty} \mathcal{II}_n =0$. It follows from the proof of Proposition \eqref{Lemma1} that
\[  \limsup_{n \rightarrow \infty } GTV_{n, k_n}(u_n)  \leq \limsup_{n \rightarrow \infty} F_n(u)  \leq \frac{\sigma_\eta}{\alpha_d^{1+1/d}} TV(u; \rho^{1-1/d}). \]
Using the density of $C_c^\infty(\R^d)$ with respect to $TV$, we may find a recovery sequence $\{ u_n \}_{n \in \N}$ with $u_n \in L^1(\nu_n)$ for every $u \in L^1(D)$.  Finally, 
if $u \in L^1(D)$ is of the form $u= \mathds{1}_A$ for some measurable set $A \subseteq D$, we can actually choose this recovery sequence to consist of indicator functions, 
as it follows from the fact that the energies $GTV_{n, k_n}$ satisfy a coarea formula (see the proof of Corollary 1.3  in \cite{GarciaTrillos2015}). This establishes the last statement in Theorem \ref{Main}.

\end{proof}

\subsection{Proof of compactness}
\begin{proof}[Proof of Theorem \ref{Compac}] The compactness in Theorem \ref{Compac} follows directly from the compactness result in \cite{GarciaTrillos2015}. Indeed, consider $\hat{\veps}_n$ as defined in the proof of the liminf inequality of Theorem \ref{Main} and define
\[  \veps_n^l := \inf_{x \in D} \hat{\veps}_n(x) - 2 \delta_n.\]
Then it follows that for all large enough $n$, 
\[  \frac{\veps_n^l}{\overline{\veps}_n} \geq c >0,\]
for some constant $c>0$. Moreover, for every $i,j$ we have
\begin{equation}
J_{k_n}(\x_i, \x_j) \geq \eta\left(  \frac{\lvert  \x_i - \x_j\rvert}{\veps_n^l} \right).
\label{IneqForCompact}
\end{equation}

Let $\{ u_n \}_{n \in \N}$ be a sequence with $u_n \in L^1(\nu_n)$ satisfying
\[  \sup_{n \in \N} \lVert u_n \rVert_{L^1(\nu_n)}  < \infty, \]
and
\[  \sup_{n \in \N} GTV_{n, k_n} (u_n)   < \infty. \]
It follows from inequality \eqref{IneqForCompact} that
\[   GTV_{n, k_n}(u_n) \geq \frac{c^d}{n^2 \veps_n^l} \sum_{i,j} \eta_{\veps_n^l}(\lvert \x_i - \x_j \rvert) \lvert u_n(\x_i) - u_n (\x_j)\rvert =: c^d GTV_{n , \veps_n^l}(u_n). \]
Therefore,
\[  \sup_{n\in \N} GTV_{n, \veps_n^l}(u_n) <\infty .\]
The pre-compactness of $\{ u_n\}_{n \in \N}$ in $TL^1$ follows from Theorem 1.2 in \cite{GarciaTrillos2015}.
\end{proof}

\appendix
\section{}

In this appendix we prove the claim we made in Step 3 in the proof of Lemma \ref{LemmaAux}. Let us start by noting that because $A$ has finite perimeter, we may talk about its \textit{reduced boundary} $\partial ^* A$ (see \cite{Maggi}). In particular, for every measurable set $S \subseteq D$,
\[ \lvert D u \rvert(S) = \mathcal{H}^{d-1}( \partial^* A \cap S), \]
where $\mathcal{H}^{d-1}$ stands for the $(d-1)$-dimensional Hausdorff measure. 

De Giorgi's structure theorem (see Theorem 15.9 in \cite{Maggi}) implies that there exist countably many
$C^1$-hypersurfaces $\{ M_i \}_{i \in \N}$, compact sets $K_i \subseteq M_i$, and a $\mathcal{H}^{d-1}$-null set $F$ such that
\[ \partial ^* A = F \cup \bigcup_{i \in \N} K_i.\]

Fix $j \in \N$. From the fact that the sets $K_i$ are subsets of $C^1$-hypersurfaces, it is straightforward to see that one can find a  sequence $\{ \tilde{\mathcal{F}}_l^j \}_{l \in \N}$ of families of closed balls contained in $D$, satisfying (1), (2), (3) as above and such that
\[  \lim_{l\rightarrow \infty}   \mathcal{H}^{d-1}\left(    \bigcup_{B \in \tilde{\mathcal{F}}_l^j } B   \cap \bigcup_{i=1}^j K_i \right) = \mathcal{H}^{d-1} \left(D\cap  \bigcup_{i=1}^j K_i \right).  \]
Still working with $j$ fixed, we now proceed to modify the balls in the families $\tilde {\mathcal{F}}_{l}^j$ to obtain balls that also satisfy (4). 
Indeed, we consider the family $\mathcal{F}_l^j$ of closed balls obtained from the balls in $\tilde{\mathcal{F}}_{l}^j$ by adding $\delta_l$ units to their original radius. The number $\delta_l$ can be chosen in such a way 
that $0<\delta_l < \frac{1}{l}$, the balls in $\mathcal{F}_l^j$ are all contained in $D$, are pairwise disjoint and satisfy property (4). This can be done because $\partial^* A \cap D$ 
has finite $\mathcal{H}^{d-1}$-measure and hence its intersection with the boundary of arbitrary balls centered at a fixed point, can only have non-zero $\mathcal{H}^{d-1}$-measure for at most countably many of such balls.

A sequence of families  $ \{ \mathcal{F}_l \}_{l \in \N}$ of balls with the desired properties can then be obtained from the families $\{ \mathcal{F}_l^j \}_{l \in \N}$ for every $j \in \N$, by using a diagonal argument knowing that 
\[  \lim_{j\rightarrow \infty} \mathcal{H}^{d-1} \left(D\cap  \bigcup_{i=1}^j K_i \right) = \mathcal{H}^{d-1} \left(D\cap  \partial^* A \right).   \]

\nc

\bibliography{kNNBIB}
\bibliographystyle{siam}

\Addresses

\end{document}